\documentclass[reqno]{amsart}
\usepackage{amsmath}
\usepackage{amssymb, hyperref}

\usepackage{graphicx,color}
\begin{document}
\title[On the IBNLS equation]{On the inhomogeneous biharmonic nonlinear Schr\"odinger equation: local, global and stability results}
	\author[C. M. GUZM\'AN ]
	{CARLOS M. GUZM\'AN }  
	
	\address{CARLOS M. GUZM\'AN \hfill\break
		Department of Mathematics, Fluminense Federal University, BRAZIL}
	\email{carlos.guz.j@gmail.com}
	
    \author[A. PASTOR ]
	{ADEMIR PASTOR}  
	
	\address{ADEMIR PASTOR  \hfill\break
	Imecc-Unicamp,	Rua S\'ergio Buarque de Holanda, 651, 13083-859, Campinas--SP, Brazil.}
	\email{apastor@ime.unicamp.br}

\begin{abstract}
We consider the inhomogeneous biharmonic nonlinear Schr\"odinger equation (IBNLS) 
$$
i u_t +\Delta^2 u+\lambda|x|^{-b}|u|^\alpha u = 0, 
$$
where $\lambda=\pm 1$ and $\alpha$, $b>0$. We show local and global well-posedness in $H^s(\mathbb{R}^N)$ in the $H^s$-subcritical case, with $s=0,2$. Moreover, we prove a stability result in $H^2(\mathbb{R}^N)$, in the mass-supercritical and energy-subcritical case. The fundamental tools to prove these results are the standard Strichartz estimates related to the linear problem.
\end{abstract}

	
	\maketitle  
	\numberwithin{equation}{section}
	\newtheorem{theorem}{Theorem}[section]
	\newtheorem{proposition}[theorem]{Proposition}
	\newtheorem{lemma}[theorem]{Lemma}
	\newtheorem{corollary}[theorem]{Corollary}
	\newtheorem{remark}[theorem]{Remark}
	\newtheorem{definition}[theorem]{Definition}

\section{Introduction}
\indent  In this paper, we study the initial value problem (IVP) associated to the inhomogeneous biharmonic nonlinear Schr\"odinger equation (IBNLS for short)
\begin{equation}\label{IBNLS}
\begin{cases}
i\partial_tu +\Delta^2 u + \lambda|x|^{-b} |u|^\alpha u =0,  \;\;\;t\in \mathbb{R} ,\;x\in \mathbb{R}^N,\\
u(0,x)=u_0(x), 
\end{cases}
\end{equation}
where $u = u(t,x)$ is a complex-valued function in space-time  $\mathbb{R}\times\mathbb{R}^N$, $\Delta^2$ stands for the biharmonic (or bilaplacian) operator, $\lambda=\pm 1$ and $\alpha, b>0$ are real numbers. The equation is called ``focusing IBNLS" when $\lambda= -1$ and ``defocusing IBNLS" when $\lambda= 1$. 

The limiting case $b = 0$ (classical biharmonic nonlinear Schr\"odinger equation (BNLS), also called the fourth-order Schr\"odinger equation)  was introduced by Karpman \cite{karpman1996} and Karpman-Shagalov \cite{karpman-Shagalov} to take into account the role of small fourth-order dispersion terms in the propagation of intense laser beams in a bulk medium with a Kerr nonlinearity. Since then, the IVP \eqref{IBNLS} (with $b=0$) has been the subject of
intensive work in recent years.  Let us recall some results:  it is known that \eqref{IBNLS}  is locally well-posed in the energy space $H^2(\mathbb{R}^N)$ in the energy-subcritical case ($0 < \alpha < \frac{8}{N-4}$, if $N \geq 5$ and $0 < \alpha < \infty$ if $1\leq N \leq 4$) and in $L^2(\mathbb{R}^N)$ in the mass-subcritical case $(0 <\alpha < \frac{8}{N})$; for details see  \cite{FIBI1} and \cite{Pausader07}. Moreover, in the defocusing case, Pausader \cite{Pausader07} studied the global well-posedness and scattering in the energy-critical case ($\alpha =\frac{8}{N-4}, \; N\geq 5$) and radially symmetric initial data. He combined the concentration-compactness argument due to Kenig-Merle \cite{KENIG} with some Morawetz-type estimates. Later, Miao-Xu-Zhao \cite{Miao-Xu-Zhao} showed a similar result  removing the radial assumption on the initial data, for $N \geq 9$. In \cite{Pausader-cubica}, Pausader showed the global well-posedness and scattering for the cubic BNLS  ($\alpha=2$) and $5\leq N \leq 8$. Furthermore, Pausader-Xia \cite{Pausader-Xia} treated the global well-posedness and scattering in the mass-supercritical case  ($\alpha >\frac{8}{N}$) and low dimensions $1 \leq N \leq 4$; they used a virial-type estimate instead of the Morawetz estimates. For the focusing case, Pausader \cite{Pausader09} and Miao-Xu-Zhao \cite{Miao-Xu-Zhao09} independently showed the global well-posedness and scattering in  the energy-critical case, assuming radially symmetric initial data with $\dot{H}^2(\mathbb{R}^N)$ and energy norms below that of the ground states. For sufficiently small initial data, Hayashi, Mendez-Navarro and Naumkin \cite{Hayashi-Nakao-Mendez} proved the global existence and the scattering for  $N = 1$ and $\alpha > 4$. They also shown the small data global existence and the decay estimates  under the assumption that the initial data is odd. Finally, we also quote Aoki, Hayashi and Naumkin \cite{Aoki-Kazuki-Hayashi-Nakao}, where the authors showed the global existence and scattering for $N = 1, 2$ and $\alpha >\frac{4}{N}$.

The equation in \eqref{IBNLS} has a counterpart for the Laplacian operator, namely, the inhomogeneous nonlinear Schr\"odinger equation  (INLS)
\begin{equation}\label{INLS}
i\partial_tu +\Delta u + \lambda|x|^{-b} |u|^\alpha u =0.
\end{equation}
In the sense of strong solutions introduced in \cite{CAZENAVEBOOK}, the well-posedness of the IVP associated with \eqref{INLS} was studied in \cite{GENSTU}, where the authors showed local well-posedness in $H^1(\mathbb{R}^N)$ for $0<b<\min\{2,N\}$ and $0<\alpha<\frac{4-2b}{N-2}$ if $N\geq3$; $0<\alpha<\infty$ if $N=1,2$. They also established global well posedness in the mass-subcritical case, that is, $0<\alpha<\frac{4-2b}{N}$. In the mass-critical case, $\alpha=\frac{4-2b}{N}$, Genoud in \cite{GENOUD} showed global well-posedness in $H^1(\mathbb{R}^N)$, provided that the mass of the initial data is below that of the associated ground state. This result was extended in the case $\frac{4-2b}{N}<\alpha<\frac{4-2b}{N-2}$ by Farah in \cite{LG}. Recently, the first author in \cite{CARLOS}, by using the contraction mapping principle combined with Strichartz estimates obtained local well-posedness results for the IVP associated to \eqref{INLS} under some restrictions on the parameters $b$ and $\alpha$; small data global theory was also established. 
Afterwards, scattering, norm concentration in the $L^2$-critical case, orbital stability of ground states and other issues were also studied (see, \cite{Mykael},  \cite{Dinh1},  \cite{paper2}, \cite{paper3}). 

Other works involving INLS model with potential, also were studied, see for instance, \cite{guo2018scattering}, \cite{Cho-Lee}. Related to IBNLS model, Cho-Ozawa-Wang \cite{Cho-Ozawa} considered the inhomogeneous power type $|x|^{-2}|u|^{\frac{4}{N}}u$. They showed the existence of weak solutions by  regularizing the nonlinearity; finite time blow-up of solutions when the energy is negative were also addressed. In some sense, by using the Strichartz estimates, we extend their result to nonlinearities of the form $|x|^{-b}|u|^{\alpha}u$. 

Our primary goal in this manuscript is to establish local and global results for the IVP \eqref{IBNLS} in $H^s(\mathbb{R}^N)$, with $s=0,2$. To this end, we use the contraction mapping argument based on the Strichartz estimates related to the linear problem. As usual, the main idea is to construct a closed subspace of $C\left([-T,T];H^s(\mathbb{R}^N)\right)$  such that the integral operator defined by 
\begin{equation}\label{OPERATOR} 
G(u)(t)=e^{it\Delta^2}u_0+i\lambda \int_0^t e^{i(t-t')\Delta^2}|x|^{-b}|u(t')|^\alpha u(t')dt'
\end{equation}
is a contraction in this subspace.  Here and in what follows, $e^{it\Delta^2}u_0$ denotes the solution to the linear problem associated with \eqref{IBNLS}.

\ Note that the IBNLS equation is invariant under the scaling,
$u_\mu(t,x)=\mu^{\frac{4-b}{\alpha}}u(\mu^4 t,\mu x)$,  $\mu >0$. This means if $u$  is a solution of \eqref{IBNLS}, with initial data $u_0$, so is $u_\mu$ with initial data $u_{\mu,0}=\mu^{\frac{4-b}{\alpha}}u_0(\mu x)$.  A straightforward computation yields
$$
\|u_{0,\mu}\|_{\dot{H}^s}=\mu^{s-\frac{N}{2}+\frac{4-b}{\alpha}}\|u_0\|_{\dot{H}^s},
$$
implying that the scale-invariant Sobolev space is $\dot{H}^{s_c}(\mathbb{R}^N)$, with $s_c=\frac{N}{2}-\frac{4-b}{\alpha}$, the so called \textit{critical Sobolev index}. If $s_c = 0$ (equivalently $\alpha = \frac{8-2b}{N}$) the IVP is known as mass-critical or $L^2$-critical; if $s_c=2$ (equivalently $\alpha =\frac{8-2b}{N-4}$) it is called energy-critical or $\dot{H}^2$-critical; also, if $s_c<0$ (equivalently $0<\alpha<\frac{8-2b}{N}$) it is called mass-subcritical or $L^2$-subcritical and  if $0<s_c<2$ (equivalently $\frac{8-2b}{N}<\alpha<\frac{8-2b}{N-4}$)  the IVP is known as mass-supercritical and energy-subcritical (or intercritical). From the above considerations it is also expected that $b$ must belong to the interval $(0,4)$.

It is well known that (at least formally) the IBNLS equation has the following conserved quantities:
\begin{equation}\label{mass}
Mass\equiv M[u(t)]=\int_{\mathbb{R}^N}|u(t,x)|^2dx=M[u_0]
\end{equation}
and
\begin{equation}\label{energy}
Energy \equiv E[u(t)]=\frac{1}{2}\int_{\mathbb{R}^N}| \Delta u(t,x)|^2dx+\frac{\lambda}{\alpha +2} \int_{\mathbb{R}^N} |x|^{-b}|u|^{\alpha +2}dx=E[u_0].
\end{equation}

\ Our interest in this paper is, in some sense, to extend some of the above mentioned results for the IBNLS model. To this end, we divide our results into three parts. The first part is devoted to study the local theory. We start considering the local well-posedness in $L^2(\mathbb{R}^N)$ and prove the following (for the precise notation see Section \ref{sec2}). 
\begin{theorem}\label{LWPL2}
Let $0< \alpha<\frac{8-2b}{N}$ and $0<b<\min\{4,N\}$, then for all $u_0 \in L^2(\mathbb{R}^N)$ there exist $T=T(\|u_0\|_{L^2},N,\alpha)>0$ and a unique solution $u$ of \eqref{IBNLS} satisfying
$$
u \in C\left([-T,T];L^2(\mathbb{R}^N)\right) \cap L^q\left([-T,T];L^{r}(\mathbb{R}^N)\right), 
$$
for any pair ($q,r$) $B$-admissible. Moreover, the continuous dependence upon the initial data holds.
\end{theorem}
 
\ Next, we deal with the local well-posedness in $H^2(\mathbb{R}^N)$. Before stating the theorem, we define the following number
\begin{equation}\label{def4*}
 4^*:=
 \begin{cases}
 \frac{8-2b}{N-4},\;\;\mbox{if}\;\;N\geq 5,\\
 +\infty,\;\;\mbox{if}\;\;1\leq N\leq 4.
 \end{cases}
\end{equation}
Note that in dimensions $N\geq5$, $\alpha=4^*$ is nothing but the index for which the IVP \eqref{IBNLS} is energy-critical.

\begin{theorem}\label{LWPH2}
Assume $N\geq3$, $0<b<\min\left\{\frac{N}{2},4\right\}$, and $\max\left\{0,\frac{2(1-b)}{N}\right\}<\alpha<4^*$. If $u_0 \in H^2(\mathbb{R}^N)$, then there exists $T=T(\|u_0\|_{H^2},N,\alpha,b)$ and a unique solution of \eqref{IBNLS} satisfying
$$
u \in C\left([-T,T];H^2(\mathbb{R}^N) \right) \cap L^q\left([-T,T];H^{2,r}(\mathbb{R}^N)    \right),
$$
where ($q,r$) is any $B$-admissible pair. Moreover, the continuous dependence upon the initial data holds.
\end{theorem}

As we already said, the proofs of Theorems \ref{LWPL2} and \ref{LWPH2} rely on the contraction mapping principle combined with the Strichartz estimates. In view of the singular factor $|x|^{-b}$ in the nonlinearity, in order to obtain the nonlinear estimates, we frequently need to divide  them inside and outside the unit ball (see Section \ref{Sec3} for details). This in turn brings some extra difficulty because we need to play with the admissible pairs along each estimate (see also \cite{Stochastic}). 

If $b<1$, then we have a lower bound for the parameter $\alpha$ in Theorem \ref{LWPH2}. This bound appears in the estimates outside the unit ball (see proof of Lemma \ref{LLH21}). On the other hand, if $b\geq1$ we then have the following.

\begin{corollary}\label{LWPH21}
	Assume $N\geq3$, $1\leq b<\min\left\{\frac{N}{2},4\right\}$, and $0<\alpha<4^*$. If $u_0 \in H^2(\mathbb{R}^N)$, then there exists $T=T(\|u_0\|_{H^2},N,\alpha,b)$ and a unique solution of \eqref{IBNLS} satisfying
	$$
	u \in C\left([-T,T];H^2(\mathbb{R}^N) \right) \cap L^q\left([-T,T];H^{2,r}(\mathbb{R}^N)    \right),
	$$
	where ($q,r$) is any $B$-admissible pair. Moreover, the continuous dependence upon the initial data holds.
\end{corollary}

\ In the second part of the paper, we consider the global well-posedness of  \eqref{IBNLS}. We begin with a global result in $L^2(\mathbb{R}^N)$, which is an immediate consequence of Theorem \ref{LWPL2} and the conservation of the mass.  

\begin{theorem}\label{GWPL2}
If $0< \alpha<\frac{8-2b}{N}$ and $0<b<\min\{4,N\}$, then for all $u_0 \in L^2(\mathbb{R}^N)$ the local solution $u$ of  \eqref{IBNLS} obtained in Theorem \ref{LWPL2} extends globally-in-time with
$$
u \in C\left(\mathbb{R};L^2(\mathbb{R}^N)\right) \cap L^q_{loc}\left(\mathbb{R};L^{r}(\mathbb{R}^N)\right), 
$$
for any  $B$-admissible pair ($q,r$). 
\end{theorem}

\ In the sequel we establish some global results in $H^2(\mathbb{R}^N)$. The first result concerns the global existence in the $L^2$-subcritical and $L^2$-critical regimes.

\begin{proposition}\label{glosub}
Assume $N\geq3$ and and $0<b<\min\left\{\frac{N}{2},4\right\}$.
Then the local solution obtained in Theorem \ref{LWPH2} can be extended globally-in-time if one of the following alternatives holds:
\begin{itemize}
	\item[(i)] $\max\left\{0,\frac{2(1-b)}{N}\right\}<\alpha<\frac{8-2b}{N}$; or
	\item[(ii)] $\alpha=\frac{8-2b}{N}$ and $\|u_0\|_{L^2}$ sufficiently small.
\end{itemize}
\end{proposition}

The proof of Proposition \ref{glosub} is an immediate consequence of the conservation of the energy and the embedding $H^2(\mathbb{R}^N)\hookrightarrow L^{\alpha+2}(|x|^{-b}dx)$. The restriction on $\alpha$ in (i) comes, of course, from the local well-posedness; once one obtains the local well-posedness for $\alpha>0$ (with $0<b<1$), then global well-posedness holds for any $\alpha>0$. 

As already commented, in \cite{GENOUD}, the author proved a similar result as in (ii) for the $L^2$-critical INLS. More precisely, he proved if $\|u_0\|_{L^2}< \|Q\|_{L^2}$, where $Q$ is the ground state solution associated with \eqref{INLS} then the solution is global in $H^1(\mathbb{R}^N)$. We believe a similar result also holds in our case; this is currently under investigation.

 Next, concerning the intercritical case we establish the following.

\begin{theorem}\label{GWPH2}
Assume one of the following conditions:
\begin{itemize}
\item[(i)] $N\geq8$, $0<b<4$, and $\frac{8-2b}{N}<\alpha<4^*$;
\item[(ii)] $N=5,6,7$, $\frac{8-2b}{N}<\alpha<\frac{N-2b}{N-4}$ and $0<b<\frac{N^2-8N+32}{8}$;
\item[(iii)] $N=6,7$, $0<b<N-4$, and	$\frac{8-2b}{N}<\alpha<4^*$;
\item[(iv)] $N=3,4$, $0<b<\frac{N}{2}$, and $\frac{8-2b}{N}<\alpha<\infty$.
\end{itemize}	
Suppose $u_0 \in H^2(\mathbb{R}^N)$ satisfies $\|u_0\|_{H^2}\leq \eta$, for some $\eta>0$. Then there exists $\delta=\delta(\eta)>0$ such that if $\|e^{it\Delta^2}u_0\|_{B(\dot{H}^{s_c})}<\delta$, then  there exists a unique global solution $u$ of \eqref{IBNLS} such that
\begin{equation*}\label{NGWP3}
\|u\|_{B(\dot{H}^{s_c})}\leq  2\|e^{it\Delta^2}u_0\|_{B(\dot{H}^{s_c})} 
\end{equation*}
and
\begin{equation*}\label{NGWP4}
\|u\|_{B\left(L^2\right)}+\|\Delta  u\|_{B\left(L^2\right)}\leq 2c\|u_0\|_{H^2},
\end{equation*}
for some universal constant $c>0$.
\end{theorem}

A few words of explanation concerning Theorem \ref{GWPH2} are in order. Its proof also relies on the contraction mapping principle. The main difficulty again is to establish the nonlinear estimates. In higher dimensions, that is, $N\geq8$ we obtain the best possible result, in the sense that $\alpha$ and $b$ range in the largest possible intervals.  Although we believe this result is also true  in other dimensions, we are unable to prove it. Note that in (ii) we need the stronger assumption $\alpha <\frac{N-2b}{N-4}$ instead of $\alpha <\frac{8-2b}{N-4}$; the assumption $0<b<\frac{N^2-8N+32}{8}$ then appears in order to have  $\frac{8-2b}{N}<\frac{N-2b}{N-4}$.  On the other hand, at least in dimension $N=6,7$, if we insist with the assumption $\alpha <\frac{8-2b}{N-4}$ then we need to impose $b<N-4$ (see also Remark \ref{remn=5} below). Finally, if $u_0\in H^2(\mathbb{R}^N)$ is such that $\|u_0\|_{\dot{H}^{s_c}}\leq\frac{\delta}{c}$, where $c$ is the constant appearing in inequality \eqref{SE2}, then we automatically have
$\|e^{it\Delta^2}u_0\|_{B(\dot{H}^{s_c})}<\delta$. In particular, if $\|u_0\|_{H^2}$ is sufficiently small, the embedding $H^2(\mathbb{R}^N)\hookrightarrow H^{s_c}(\mathbb{R}^N)$ (recall we are in the case $0<s_c<2$) gives that $\|u_0\|_{H^{s_c}}$ is also sufficiently small; hence, we deduce the existence of a global solution  if one of the conditions (i)-(iv) hold.

Once global results are established, the natural issue is to study the asymptotic behavior of such global solutions as $t\rightarrow \pm  \infty$. Here we shall show that our solutions scatters to a solution of the linear problem.

 \begin{proposition}\label{SCATTERSH1}{\bf ($H^2$ Scattering)}  Let $u(t)$ be  a global solution of \eqref{IBNLS} with initial data $u_0 \in  H^2(\mathbb{R}^N)$. Suppose  $\|u\|_{B(\dot{H}^{s_c})}< +\infty$ and $\sup\limits_{t\in \mathbb{R}}\|u(t)\|_{H^2_x}\leq \eta$. If one of the assumptions (i)-(iv) in Theorem \ref{GWPH2} hold, then $u(t)$ scatters in $H^2(\mathbb{R}^N)$ as $t \rightarrow \pm\infty$. More precisely, there exists  $\phi^{\pm}\in H^2(\mathbb{R}^N)$ such that
$$
\lim_{t\rightarrow \pm\infty}\|u(t)-e^{it\Delta^2}\phi^{\pm}\|_{H^2}=0.
$$
\end{proposition}

Note that Proposition \ref{SCATTERSH1} gives a suitable criterion to establish the scattering of a global solution. Is is clear that we do not need to assume that $\|e^{it\Delta^2}u_0\|_{B(\dot{H}^{s_c})}$ is small. However, Proposition \ref{SCATTERSH1} immediately gives the scattering of small solutions. More precisely,

\begin{corollary}
Assume that assumptions in Theorem \ref{GWPH2} hold. If $\eta$ is sufficiently small then the unique global solution  scatters in $H^2(\mathbb{R}^N)$.
\end{corollary}

We believe the existence of global solutions and scattering, in the intercritical case, may be obtained for large initial data if they satisfy a suitable balance between the mass and the energy. For the NLS equation this was already obtained, for instance, in \cite{JIANCAZENAVE} and \cite{HOLROU}. In the case of the INLS this was recently obtained in \cite{LG}.  This is also under investigation.

The last part of this work is devoted to study stability of the solutions of \eqref{IBNLS}, in the intercritical case ($0<s_c<2$). By
stability we mean if we have an approximate solution to \eqref{IBNLS}, as in \eqref{appsol}, 
with an $e$ small in a suitable norm and $\widetilde{u}_0-u_0$ small in $\dot{H}^{s_c}$, then there exists a solution $u$ to \eqref{IBNLS} which stays  close to  $\widetilde{u}$ in critical norms. More precisely,

\begin{theorem}\label{LTP} 
Assume that assumptions in Theorem 1.6 hold. Let $I\subseteq \mathbb{R}$ be a time interval containing zero. Let $\widetilde{u}$  be a solution to 
\begin{equation}\label{appsol}
i\partial_t \widetilde{u} +\Delta^2 \widetilde{u} + \lambda |x|^{-b} |\widetilde{u}|^\alpha \widetilde{u} =e,
\end{equation}  
defined on $I\times \mathbb{R}^N$, with initial data $\widetilde{u}_0\in H^2(\mathbb{R}^N)$, satisfying (for some positive constants $M,L$)
\begin{equation}\label{HLP1} 
\sup_{t\in I}  \|\widetilde{u}\|_{H^2_x}\leq M \;\; \textnormal{and}\;\; \|\widetilde{u}\|_{B(\dot{H}^{s_c}; I)}\leq L.
\end{equation}
 
\indent Let $u_0\in H^2(\mathbb{R}^N)$ such that 
\begin{equation}\label{HLP2}
\|u_0-\widetilde{u}_0\|_{H^2}\leq M'\;\; \textnormal{and}\;\; \|e^{it\Delta^2}(u_0-\widetilde{u}_0)\|_{B(\dot{H}^{s_c}; I)}\leq \varepsilon,
\end{equation}
for some positive constant $M'$ and some $0<\varepsilon<\varepsilon_1=\varepsilon_1(M,M',L)$. In addition, assume also the following conditions
\begin{equation*}
\|e\|_{B'(\dot{H}^{s_c}; I)}+\| e\|_{B'(L^2; I)}+ \|\nabla e\|_{L^2_IL^{\frac{2N}{N+2}}_x}\leq \varepsilon.
\end{equation*}
\indent Then, there exists a unique solution $u$ to \eqref{IBNLS} on $I\times \mathbb{R}^N$, with  $u(0)=u_0$, satisfying 
\begin{equation}\label{CLP} 
\|u-\widetilde{u}\|_{B(\dot{H}^{s_c}; I)}\leq C(M,M',L)\varepsilon\;\;\;\;\;\;\;\textnormal{and}
\end{equation}
\begin{equation}\label{CLP1}
\|u\|_{B(\dot{H}^{s_c}; I)} +\|u\|_{B(L^2; I)}+\|\Delta u\|_{B(L^2; I)}\leq C(M,M',L).
\end{equation}
\end{theorem}

The proof o Theorem \ref{LTP} also relies on the estimates presented in Section \ref{secglo}. Note that the case $e=0$ corresponds to the question of continuous dependence upon the data.

\ The rest of the paper is organized as follows. In section \ref{sec2}, we introduce some notations and give a review of the Strichartz estimates. In Section \ref{Sec3}, we prove the local well-posedness results. In Section \ref{secglo}, we prove the results concerning the global theory as well as the scattering one.  The final section, Section $5$, is devoted to study the stability theory.

\section{\bf Notation and Preliminaries}\label{sec2}

In this section, we introduce the notation used throughout the paper and list some useful results.  We use $c$ to denote various constants that may vary line by line. Let $a$ and $b$ be positive real numbers, the
notation $a \lesssim b$ means that there exists a positive constant $c$ such that $a \leq cb$. Given a real number $r$, we use $r+$ to denote $r+\varepsilon$ for some $\varepsilon>0$ sufficiently small. For a subset $A\subset \mathbb{R}^N$, $A^C=\mathbb{R}^N \backslash A$ denotes the complement of $A$. Given $x,y \in \mathbb{R}^N$, $x \cdot y$ denotes the usual inner product of $x$ and $y$ in $\mathbb{R}^N$.

The norm in the Sobolev spaces $H^{s,r}=H^{s,r}(\mathbb{R}^N)$ and $\dot{H}^{s,r}=\dot{H}^{s,r}(\mathbb{R}^N)$, are defined, respectively, by $\|f\|_{H^{s,r}}:=\|J^sf\|_{L^r}$ and $\|f\|_{\dot{H}^{s,r}}:=\|D^sf\|_{L^r},$
where  $J^s$ and $D^s$ stand for the Bessel and Riesz potentials of order $s$, given via Fourier transform by $\widehat{J^s f}=(1+|\xi|^2)^{\frac{s}{2}}\widehat{f}$ and $\widehat{D^sf}=|\xi|^s\widehat{f}.
$
If $r=2$ we denote $H^{s,2}$ and $\dot{H}^{s,2}$ simply by $H^s$ and  $\dot{H}^{s}$, respectively.

Let $ q,r >0$, $s\in \mathbb{R}$, and $I\subset \mathbb{R}$ an interval; the mixed norms in the spaces $L^q_{I}L^r_x$ and $L^q_{I} H^s_x$ of a function $f=f(t,x)$ are defined as
$$
\|f\|_{L^q_{I}L^r_x}=\left(\int_I\|f(t,\cdot)\|^q_{L^r_x}dt\right)^{\frac{1}{q}}
\qquad
\mbox{and}
\qquad
\|f\|_{L^q_{I}H^s_x}=\left(\int_I\|f(t,\cdot)\|^q_{H^s_x}dt\right)^{\frac{1}{q}},
$$
with the usual modifications if either $q=\infty$ or $r=\infty$. When the  $x$-integration is restricted to a subset $A\subset\mathbb{R}^N$ then the mixed norm will be denoted by $\|f\|_{L_I^qL^r_x(A)}$. Moreover, if $I=\mathbb{R}$ we shall use the notations $\|f\|_{L_t^qL^r_x}$ and $\|f\|_{L_t^qH^s_x}$.

Next, we recall the Sobolev inequalities.
\begin{lemma}[\textbf{Sobolev embedding}]\label{SI} Let $s\in (0,+\infty)$ and $1\leq p<+\infty$.
\begin{itemize}
\item [(i)] If $s\in (0,\frac{N}{p})$ then $H^{s,p}(\mathbb{R}^N)$ is continuously embedded in $L^r(\mathbb{R^N})$ where $s=\frac{N}{p}-\frac{N}{r}$. Moreover, 
\begin{equation}\label{SEI} 
\|f\|_{L^r}\leq c\|D^sf\|_{L^{p}}.
\end{equation}
\item [(ii)] If $s=\frac{N}{2}$ then $H^{s}(\mathbb{R}^N)\subset L^r(\mathbb{R^N})$ for all $r\in[2,+\infty)$. Furthermore,
\begin{equation}\label{SEI1} 
\|f\|_{L^r}\leq c\|f\|_{H^{s}}.
\end{equation}
\item [(iii)] If $s>\frac{N}{2}$ then $H^{s}(\mathbb{R}^N)\subset L^\infty(\mathbb{R^N})$.
\end{itemize}
\end{lemma}
\begin{proof} See Bergh-L\"ofstr\"om \cite[Theorem $6.5.1$]{BERLOF} (see also Linares-Ponce \cite[Theorem $3.3$]{FELGUS} and Demenguel-Demenguel \cite[Proposition 4.18]{DEMENGEL}). 
\end{proof}

\ Next, we recall some Strichartz type estimates associated to the linear biharmonic Schr\"odinger propagator. We say the pair $(q,r)$ is biharmonic Schr\"odinger admissible ($B$-admissible for short) if it satisfies
\begin{equation*}
\frac{4}{q}=\frac{N}{2}-\frac{N}{r},
\end{equation*}
with
\begin{equation}\label{L2Admissivel}
\begin{cases}
2\leq  r  < \frac{2N}{N-4},\hspace{0.5cm}\textnormal{if}\;\;\;  N\geq 5,\\
2 \leq  r < + \infty,\;  \hspace{0.5cm}\textnormal{if}\;\;\;1\leq N\leq 4.
\end{cases}
\end{equation}

Also, given a real number $s<2$, the pair $(q,r)$ is called $\dot{H}^s$-biharmonic admissible  if 
\begin{equation}\label{CPA1}
\frac{4}{q}=\frac{N}{2}-\frac{N}{r}-s
\end{equation}
with
\begin{equation}\label{HsAdmissivel}
\begin{cases}
\frac{2N}{N-2s} \leq  r  <\frac{2N}{N-4}\;\;\;  N\geq 5,\\
2 \leq  r < + \infty,\;  \hspace{0.5cm}\textnormal{if}\;\;\;1\leq N\leq 4.
\end{cases}
\end{equation}
We set $\mathcal{B}_s:=\{(q,r);\; (q,r)\; \textnormal{is} \;\dot{H}^s\textnormal{-biharmonic admissible}\}$. Also, given $(q,r)\in \mathcal{B}_s$, by $(q',r')$ we denote its dual pair, that is, $\frac{1}{q}+\frac{1}{q'}=1$ and $\frac{1}{r}+\frac{1}{r'}=1$. We define the Strichartz norm by
$$
\|u\|_{B(\dot{H}^{s})}=\sup_{(q,r)\in \mathcal{B}_{s}}\|u\|_{L^q_tL^r_x} 
$$
and the dual Strichartz norm by
$$
\|u\|_{B'(\dot{H}^{-s})}=\inf_{(q,r)\in \mathcal{B}_{-s}}\|u\|_{L^{q'}_tL^{r'}_x}.
$$
Note that, if $s=0$ then $\mathcal{B}_0$ is the set of all $B$-admissible pairs. It is to be clear that we write $B(\dot{H}^s)$ or $B'(\dot{H}^{-s})$ if the mixed norm is evaluated over $\mathbb{R}\times\mathbb{R}^N$. To indicate the restriction to a time interval $I\subset (-\infty,\infty)$ or a subset $A\subset\mathbb{R}^N$, we will use the notations $B(\dot{H}^s(A);I)$ and $B'(\dot{H}^{-s}(A);I)$. 

\ The main tools to show the local and global well-posedness are the well-known Strichartz estimates. See for instance Pausader \cite{Pausader07} (see also \cite{Guo}).

\begin{lemma}\label{Lemma-Str}
Let $I\subset\mathbb{R}$ be an interval and $t_0\in I$.
	The following statements hold.
 \begin{itemize}
\item [(i)] (\textbf{Linear estimates}).
\begin{equation}\label{SE1}
\| e^{it\Delta^2}f \|_{B(L^2;I)} \leq c\|f\|_{L^2},
\end{equation}
\begin{equation}\label{SE2}
\| e^{it\Delta^2}f \|_{B(\dot{H}^s;I)} \leq c \|f\|_{\dot{H}^s}.
\end{equation}
\item[(ii)] (\textbf{Inhomogeneous estimates}).
\begin{equation}\label{SE3}					 
\left \| \int_{t_0}^t e^{i(t-t')\Delta^2}g(\cdot,t') dt' \right \|_{B(L^2;I) } \leq c\|g\|_{B'(L^2;I)},
\end{equation}
\begin{equation}\label{SE5}
\left \| \int_{t_0}^t e^{i(t-t')\Delta^2}g(\cdot,t') dt' \right \|_{B(\dot{H}^s;I) } \leq c\|g\|_{B'(\dot{H}^{-s};I)}.
\end{equation}
\end{itemize}
\end{lemma} 

Finally, we list other useful Strichartz estimates for the fourth-order Schr\"odinger equation. Recall that a pair $(q,r)$ is called Schr\"odinger admissible ($S$-admissible for short) if $ 2\leq q, r \leq\infty$, $(q,r,N)\neq (2,\infty, 2)$, and
\begin{equation*}
	\frac{2}{q}=\frac{N}{2}-\frac{N}{r}.
\end{equation*}

\begin{proposition}\label{estimativanaolinear} Let $I\subset\mathbb{R}$ be an interval and $t_0\in I$. Suppose that  $s \geq 0$ and $u\in C(I,H^{-4})$ is a solution of
$$
u(t)= e^{i(t-t_0)\Delta^2}u(t_0)+i\lambda \int_{t_0}^t e^{i(t-t')\Delta^2}F(\cdot,t')dt',
$$
for some function $F\in L^1_{loc}(I,H^{-4})$. Then,
\begin{itemize}
	\item[(i)] For any $S$-admissible pairs $(m, n)$ and $(a, b)$, we have
	\begin{equation}\label{ESB1}
	\left\|D^s u\right\|_{L^{m}_{I}L_x^{n}} \lesssim \left\| D^{s-\frac{2}{m}}u(t_0)\right\|_{L^{2}}+\left\|D^{s-\frac{2}{m}-\frac{2}{a}}F \right\|_{L^{a'}_{I}L_x^{b'}}.
	\end{equation}
	\item[(ii)] If $N \geq 3$ then for any $B$-admissible pair $(q,r)$, we obtain 
	\begin{equation}\label{ESB2}
	\left\|D^s u\right\|_{L^{q}_{I}L_x^{r}} \lesssim \left\| D^{s}u(t_0)\right\|_{L^{2}}+\left\|D^{s-1} F\right\|_{L^{2}_{I}L_x^{\frac{2N}{N+2}}},
	\end{equation}
	 \end{itemize} 

In particular, when $s=2$, \eqref{ESB1} writes as 
\begin{equation}\label{ESB3}
\left\|\Delta u\right\|_{L^{q}_{I}L_x^{r}}\lesssim \|D^{2+\frac{2}{q}} u\|_{L^{q}_{I}L_x^{\bar{r}}}  \lesssim \left\| \Delta u(t_0)\right\|_{L^{2}}+\left\|D^{2-\frac{2}{a}}F \right\|_{L^{a'}_{I}L_x^{b'}}
\end{equation}
and \eqref{ESB2} as
\begin{equation}\label{EstimativaImportante}
\left\|\Delta u\right\|_{L^{q}_{I}L_x^{r}} \lesssim \left\| \Delta u(t_0)\right\|_{L^{2}}+\left\|\nabla F\right\|_{L^{2}_{I}L_x^{\frac{2N}{N+2}}}.
\end{equation}
\end{proposition}
\begin{proof}
For (i) see \cite[Proposition 3.1]{Pausader07}. For (ii), from Sobolev's embedding,
$$
\left\|D^s u\right\|_{L^{q}_{I}L_x^{r}} \lesssim \left\|D^{s+\frac{2}{q}} u\right\|_{L^{q}_{I}L_x^{\bar{r}}},
$$
where $\bar{r}$ is such that $\frac{2}{q}=\frac{N}{\bar{r}}-\frac{N}{r}$. Since $(q,r)$ is $B$-admissible, it is easily seen that $(q,\bar{r})$ is $S$-admissible. In addition, since $\left(2,\frac{2N}{N-2}\right)$ is also $S$-admissible with dual pair $\left(2,\frac{2N}{N+2}\right)$, the result then follows from (i).
\end{proof}

\begin{remark}
As usual, if $I=(T,+\infty)$ then in Lemma \ref{Lemma-Str} and Proposition \ref{estimativanaolinear} one may replace the integral $\int_{t_0}^t$ by $\int_t^{+\infty}$. This will be necessary in the proof of Proposition \ref{SCATTERSH1}. A similar statement holds if $I=(-\infty, T)$.
\end{remark}

Throughout the paper, $B$ will denote the unity ball in $\mathbb{R}^N$, that is, $B=\{ x\in \mathbb{R}^N;|x|\leq 1\}$. Recall that
	$$\||x|^{-b}\|_{L^\gamma(B)}<+\infty\;\;\;\textnormal{if}\;\;\frac{N}{\gamma}-b>0,
	$$
	and 
	$$\||x|^{-b}\|_{L^\gamma(B^C)}<+\infty\;\;\;\textnormal{if}\;\;\frac{N}{\gamma}-b<0.
	$$
This will be frequently used  along the paper. Finally, if
 $F(x,z)=|x|^{-b}|z|^\alpha z$,   then (see details in \cite[Remark 2.6]{CARLOS} and \cite[Remark 2.5]{paper2})
\begin{equation}\label{FEI}
 |F(x,z)-F(x,w)|\lesssim |x|^{-b}\left( |z|^\alpha+ |w|^\alpha \right)|z-w|
\end{equation}
and
\begin{equation}\label{SECONDEI}
\left|\nabla \left(F(x,z)-F(x,w)\right)\right|\lesssim  |x|^{-b-1}(|z|^{\alpha}+|w|^{\alpha})|z-w|+|x|^{-b}|z|^\alpha|\nabla (z- w)|+E, 
\end{equation}
where 
\begin{eqnarray*}
 E &\lesssim& \left\{\begin{array}{cl}
 |x|^{-b}\left(|z|^{\alpha-1}+|w|^{\alpha-1}\right)|\nabla w||z-w| & \textnormal{if}\;\;\;\alpha > 1 \vspace{0.2cm} \\ 
|x|^{-b}|\nabla w||z-w|^{\alpha} & \textnormal{if}\;\;\;0<\alpha\leq 1.
\end{array}\right.
\end{eqnarray*}

\section{\bf Local well-posedness}\label{Sec3}
				
\ In this section we prove the local well-posedness results. The theorems follow from a contraction mapping argument based on the Strichartz estimates. First, we show the local well-posedness in $L^2(\mathbb{R}^N)$ (Theorem \ref{LWPL2}) and then in $H^2(\mathbb{R}^N)$  (Theorem \ref{LWPH2}). 

\subsection{Local Well-Posedness in $L^2$} We start with the following lemma. It provides an estimate for the nonlinearity in the Strichartz spaces. 

\begin{lemma}\label{lemmaL2} 
Let $0<\alpha<\frac{8-2b}{N}$  and $0<b<\min \{4,N\}$. Then,
\begin{equation}\label{NlemmaL2} 
\left\|\chi_{B^C}|x|^{-b}|u|^{\alpha}v \right\|_{B'(L^2;I)}+\left\|\chi_{B}|x|^{-b}|u|^{\alpha}v \right\|_{B'(L^2;I)}\leq c (T^{\theta_1}+T^{\theta_2})\| u\|^{\alpha}_{B(L^2;I)}\| v\|_{B(L^2;I)},
\end{equation}
where $I=[0,T]$ and $c,\theta_1,\theta_2 >0$.
\begin{proof} We start by estimating $\left\|\chi_{B^C}|x|^{-b}|u|^{\alpha}v \right\|_{B'(L^2;I)}$. Indeed, if $x\in B^C$ then $|x|^{-b}<1$, thus
\begin{equation*}
\left\|\chi_{B^C} |x|^{-b}|u|^{\alpha}v \right\|_{B'\left(L^2;I\right)} \leq  \left\||u|^{\alpha}v
\right\|_{B'\left(L^2;I\right)}.
\end{equation*}
Let $(q,r)$ be the  $B$-admissible pair $\left(\frac{8(\alpha+2)}{N\alpha},\alpha+2\right)$. By using the H\"older inequality we have (since $\frac{1}{r'}=\frac{\alpha}{r}+\frac{1}{r}$)
\begin{align*}
\left\||u|^{\alpha}v
\right\|_{B'\left(L^2;I\right)} \leq  \left\|  |u|^{\alpha} v \right\|_{L^{q'}_IL_x^{r'}} 
\leq T^{\frac{1}{q_1}} \| u\|^{\alpha}_{L_I^{q}L_x^r}\| v\|_{L_I^{q}L_x^{r}}=T^{\theta_1} \| u\|^{\alpha}_{L_I^{q}L_x^r}\| v\|_{L_I^{q}L_x^{r}},
\end{align*}where $\theta_1=\frac{1}{q_1}=1-\frac{\alpha+2}{q}=\frac{8-N\alpha}{8}$, which is positive by our hypothesis on $\alpha$. Therefore, 
\begin{equation}\label{A1}
\left\|\chi_{B^C} |x|^{-b}|u|^{\alpha}v \right\|_{B'\left(L^2;I\right)}\leq c T^{\theta_1} \| u\|^{\alpha}_{B(L^2;I)}\| v\|_{B(L^2;I)}.
\end{equation}

\ We now consider the term $\left\|\chi_{B} |x|^{-b}|u|^{\alpha}v \right\|_{B'\left(L^2;I\right)}$. If $(q,r)$ is any $B$-admissible pair, applying H\"older's inequality, one has 
\begin{align*}
\left\|\chi_{B} |x|^{-b}|u|^{\alpha}v \right\|_{B'\left(L^2;I\right)} &\leq  \left\|\chi_{B} |x|^{-b}|u|^{\alpha} v \right\|_{L^{q'}_IL_x^{r'}} \leq \left \| \||x|^{-b}\|_{L^\gamma(B) } \|u\|^{\alpha}_{L_x^{\alpha r_1}} \|v\|_{L_x^{ r}}\right\|_{L_I^{q'}}\\
&\leq  \||x|^{-b}\|_{L^\gamma(B)}T^{\frac{1}{q_1}} \| u\|^{\alpha}_{L_I^{\alpha q_2}L_x^{\alpha r_1}}\| v\|_{L_I^{q}L_x^{r}}\\
&\leq  T^{\frac{1}{q_1}}\||x|^{-b}\|_{L^\gamma(B)} \| u\|^{\alpha}_{L_I^{q}L_x^r}\| v\|_{L_I^{q}L_x^{r}},
\end{align*}			
where
\begin{equation}\label{lemmaL21}
\begin{cases}
\frac{1}{r'}=\frac{1}{\gamma}+\frac{1}{r_1}+\frac{1}{r}, \\ 
\frac{1}{q'}=\frac{1}{q_1}+\frac{1}{q_2}+\frac{1}{q},\\
q= \alpha q_2,\;\;\;\;r=\alpha r_1.
\end{cases}
\end{equation}

Recall that $\||x|^{-b}\|_{L^\gamma(B)}$ is finite provided that $\frac{N}{\gamma}>b$. Hence, in view of \eqref{lemmaL21} the pair $(q,r)$ must satisfy
\begin{equation}\label{lemmaL22}
\begin{cases}
\frac{N}{\gamma}=N-\frac{N(\alpha+2)}{r}>b,\\
\frac{1}{q_1}=1-\frac{\alpha+2}{q}.
\end{cases}
 \end{equation}
The first inequality in \eqref{lemmaL22} is equivalent to 
\begin{equation}\label{lemmaL23}
\alpha<\frac{r(N-b)-2N}{N},
\end{equation}
for $r>\frac{2N}{N-b}$. Since $\alpha <\frac{8-2b}{N}$ we can choose $r$ such that $\frac{r(N-b)-2N}{N}=\frac{8-2b}{N}$, i.e., 
$r=\frac{8-2b+2N}{N-b}$. On the other hand, since $(q,r)$ must be $B$-admissible, it is clear that $q=\frac{8-2b+2N}{N}$. Next, it follows from the second equation in \eqref{lemmaL22} that
$$
\frac{1}{q_1}=\frac{8-2b-\alpha N}{8 -2b+2N},
$$ 
which is positive in view of our hypothesis on $\alpha$. Consequently, 
\begin{equation*}\label{A2}
\left\|\chi_{B} |x|^{-b}|u|^{\alpha}v \right\|_{B'\left(L^2;I\right)}\leq c T^{\theta_2} \| u\|^{\alpha}_{B(L^2;I)}\| v\|_{B(L^2;I)},
\end{equation*}
where $\theta_2=\frac{1}{q_1}>0$.
Combining \eqref{A1} with the last inequality we obtain \eqref{NlemmaL2}.	
\end{proof}
\end{lemma}

\ Our goal now is to show Theorem \ref{LWPL2}. 
					
\begin{proof}[\bf{Proof of Theorem \ref{LWPL2}}]  For any $B$-admissible pair $(q,r)$, define
$$X= C\left( [-T,T];L^2(\mathbb{R}^N)\right) \bigcap  L^q\left([-T,T];L^{r}(\mathbb{R}^N)\right),$$ and 
\begin{equation*}\label{BL} 
 S(a,T)=\{u \in X : \|u\|_{B\left(L^2;[-T,T]\right)}\leq a \},
\end{equation*}
where $a$ and $T$ are positive constants to be determined later. We will prove there are $a$ and $T$ such that the operator $G$ defined in \eqref{OPERATOR} acts from $S(a,T)$ to itself and is a contraction.

\ Without loss of generality we consider only the case $t>0$. The Strichartz inequalities (\ref{SE1}) and (\ref{SE3}) yield
\[
\begin{split}
\|G(u)\|_{B(L^2;I)} &\lesssim \|u_0\|_{L^2}+ \left\|\int_0^t e^{i(t-t')\Delta^2}|x|^{-b}|u|^{\alpha}u  \right\|_{B(L^2;I)} \\
&\lesssim \|u_0\|_{L^2}+ \left\|\int_0^t e^{i(t-t')\Delta^2}\chi_{B^C}|x|^{-b}|u|^{\alpha}u  \right\|_{B(L^2;I)} \\
& \quad+\left\|\int_0^t e^{i(t-t')\Delta^2}\chi_{B}|x|^{-b}|u|^{\alpha}u  \right\|_{B(L^2;I)} \\
&\lesssim \|u_0\|_{L^2}+ \left\|\chi_{B^C}|x|^{-b}|u|^{\alpha}u  \right\|_{B'(L^2;I)}+\left\|\chi_{B}|x|^{-b}|u|^{\alpha}u  \right\|_{B'(L^2;I)},
\end{split}
\]
where $I=[0,T]$. Now, for any  $u\in S(a,T)$, Lemma \ref{lemmaL2} yields
\[
\begin{split}
\|G(u)\|_{B(L^2;I)}  
&\leq c\|u_0\|_{L^2}+c (T^{\theta_1}+T^{\theta_2})\| u\|^{\alpha+1}_{B(L^2;I)}\\
&\leq c\|u_0\|_{L^2}+c (T^{\theta_1}+T^{\theta_2}) a^{\alpha+1}.
\end{split}
\]
Next, by choosing $a=2c\|u_0\|_{L^2}$ and $T>0$ such that 
\begin{equation}\label{NT3}
c a^{\alpha} (T^{\theta_1}+T^{\theta_2}) < \frac{1}{4},
\end{equation}
we conclude  $G(u)\in S(a,T)$. Similarly, in view of \eqref{FEI},
\begin{equation*}
\begin{split}
\|G(u)-G(v)\|_{B(L^2;I)}
&\leq c\left\| \chi_{B^C}|x|^{-b} |u|^{\alpha}|u-v|\right\|_{B'(L^2;I)}+c\left\| \chi_{B}|x|^{-b} |v|^{\alpha}|u-v|\right\|_{B'(L^2;I)}\\
&\leq c (T^{\theta_1}+T^{\theta_2}) \left( \|u\|^\alpha_{B(L^2;I)}+\|v\|^\alpha_{B(L^2;I)} \right) \|u-v\|_{B(L^2;I)}.
\end{split}
\end{equation*}
 Hence, if  $u,v\in S(a,T)$, inequality $\eqref{NT3}$ implies that
 \[
\begin{split}
\|G(u)-G(v)\|_{B\left(L^2;[-T,T]\right)} &\leq  2c (T^{\theta_1}+T^{\theta_2})a^\alpha \|u-v\|_{B\left(L^2;[-T,T]\right)}\\
&<  \frac{1}{2}\|u-v\|_{B\left(L^2;[-T,T]\right)},
\end{split}
\]
which means that $G$ is a contraction on $S(a,T)$. The contraction mapping principle then implies the existence of a unique solution. To finish the proof, we use standard arguments; thus
we omit the details.
\end{proof}

\subsection{Local Well-Posedness in $H^2$}
The goal of this subsection is to show the local well-posedness in $H^2(\mathbb{R}^N)$. Before doing that we establish useful estimates for the nonlinearity $F(x,u)=|x|^{-b}|u|^\alpha u$. To do so, we will use the Sobolev embedding (see Lemma \ref{SI}) according to the cases: $N\geq 5$, $N=4$ and $1\leq N\leq3$. More precisely, 
\[
H^2(\mathbb{R}^N)\hookrightarrow 
\begin{cases}
L^r(\mathbb{R}^N),\;\; \mbox{if}\;\; N\geq 5 \;\; \mbox{and} \;\; 2\leq r\leq \frac{2N}{N-4};\\
L^r(\mathbb{R}^N), \;\; \mbox{if}\;\; N=4 \;\; \mbox{and}\;\; r\geq 2;\\
L^\infty(\mathbb{R}^N), \;\;\mbox{if} \;\;N=1,2,3.
\end{cases}
\]

 Before stating the lemmas, we define the norm
\begin{equation}\label{NormaBCL}
 \|u\|_I=\| u\|_{B(L^2;I)}+\|\Delta u\|_{B(L^2;I)},
\end{equation}
where $I=[0,T]$.

\begin{lemma}\label{LLH210} Let $N\geq 1$, $0<b<\min\{\frac{N}{2},4\}$ and $0<\alpha<4^*$, where $4^*$ is defined in \eqref{def4*}. The following statement holds  
$$
\left \|\chi_{B^C}|x|^{-b}|u|^\alpha v \right\|_{B'(L^2;I)}+\left \|\chi_{B}|x|^{-b}|u|^\alpha v \right\|_{B'(L^2;I)}\leq c (T^{\theta_1}+T^{\theta_2})\|u\|^\alpha_{I} \|v\|_I,
$$
where $c,\theta_1,\theta_2 >0$.
\begin{proof} 
Let $B_1=\left\|\chi_{B^C}|x|^{-b}|u|^\alpha v\right\|_{B'(L^2;I)}$ and $B_2=\left\|\chi_{B}|x|^{-b}|u|^\alpha v\right\|_{B'(L^2;I)}$.
We consider two cases.

\ {\bf Case 1: $N\geq 5$}. First, we estimate $B_1$. Let $(q_0,r_0)$ be defined as
\begin{equation}\label{L1Hs1} 
q_0= \frac{8(\alpha+2)}{\alpha(N-4)} \;\;\; \textnormal{and} \;\; r_0=\frac{ N(\alpha+2)}{ N+ 2\alpha }.
 \end{equation}
 It is easily seen that $\frac{4}{q_0}=\frac{N}{2}-\frac{N}{r_0}$ and $r_0\geq2$. In addition, $r_0<\frac{2N}{N-4}$ is equivalent to $\alpha(N-8)<8$. This last inequality trivially holds if $N\leq8$. On the other hand, if $N>8$ our assumptions on $\alpha$ and $b$ implies $\alpha<\frac{8}{N-8}$. As a consequence, we obtain that $(q_0,r_0)$ is a $B$-admissible pair.
 
Note that $r_0<\frac{N}{2}$ (since $N>4$). Let $r_1$ be defined as
$$
\frac{N}{\alpha r_1}=\frac{N}{r_0}-2.
$$
An easy computation shows that $\frac{1}{r_0'}=\frac{1}{r_1}+\frac{1}{r_0}$. Hence, H\"older's inequality and  Sobolev's embedding \eqref{SEI}  imply
 \begin{align}\label{L1Hs11}
 B_1&\leq   \left \|\||x|^{-b}\|_{L^\infty(B^C)}\|u\|^\alpha_{L_x^{\alpha r_1}} \|v\|_{L^{r_0}_x}\right\|_{L^{q'_0}_I} \nonumber  \\
 &\leq c	\left \|  \|\Delta u\|^\alpha_{L_x^{r_0}} \|v\|_{L^{r_0}_x}\right\|_{L^{q'_0}_I} \nonumber   \\
  &\leq cT^{\frac{1}{q_1}} \|\Delta u\|^\alpha_{L_I^{q_0}L^{r_0}_x} \|v\|_{L^{q_0}_IL_x^{r_0}},\nonumber 
 \end{align}
where 
 \begin{equation*}\label{CLHs1} 
\frac{1}{q'_0}=\frac{1}{q_1}+\frac{\alpha}{q_0}+\frac{1}{q_0}.
\end{equation*}
Taking into account the definition of $q_0$ in \eqref{L1Hs1}, we deduce
\begin{equation}\label{CLHs12}
\frac{1}{q_1}= 1-\frac{\alpha+2}{q_0}= \frac{8-\alpha(N-4)}{8},
\end{equation}
which is positive by our hypothesis $\alpha<4^*$. Therefore, setting $\theta_1=\frac{1}{q_1}$ we deduce
\begin{equation}\label{LHsB1}
B_1 \leq c T^{\theta_1}\|\Delta u\|^\alpha_{B(L^2;I)} \|v\|_{B(L^2;I)}\leq c T^{\theta_1}\|u\|^\alpha_{I} \|v\|_I.
\end{equation} 	

 \ We now estimate $B_2$. To do this, we need to divide the argument into two cases.
 
 \ {\bf Case 1.1: $N\geq 8$}. Let $(q,r)$ be any $B$-admissible pair.
 It follows from H\"older's inequality and Sobolev embedding that
 \begin{equation}\label{a1}
\begin{split}
B_2&\leq \left\|\chi_B|x|^{-b}|u|^\alpha v\right\|_{L^{q'}_IL_x^{r'}} \leq \left \|\||x|^{-b}\|_{L^\gamma(B)} \|u\|^\alpha_{L_x^{\alpha r_1}} \|v\|_{L^r_x}\right\|_{L^{q'}_I}\\
&\leq \left \|\||x|^{-b}\|_{L^\gamma(B)} \|\Delta u\|^\alpha_{L_x^{r}} \|v\|_{L^r_x}\right\|_{L^{q'}_I}\\
&=  \||x|^{-b}\|_{L^\gamma(B)}T^{\frac{1}{q_1}} \|\Delta u\|^\alpha_{L_I^{q}L^r_x} \|v\|_{L^{q}_IL_x^r},
\end{split}
\end{equation}
where
\begin{equation}\label{CLHs2} 
\begin{cases}
\frac{1}{r'}=\frac{1}{\gamma}+\frac{1}{r_1}+\frac{1}{r},\\ 
2=\frac{N}{r}-\frac{N}{\alpha r_1}, \quad r<\frac{N}{2}, \\ 
\frac{1}{q'}=\frac{1}{q_1}+\frac{\alpha}{q}+\frac{1}{q}.
\end{cases}
\end{equation}
 From \eqref{CLHs2} we see that
\begin{equation}\label{CLHs3} 
\begin{cases}
\frac{N}{\gamma}= N-\frac{2N}{r}-\frac{N\alpha }{r}+2\alpha, \\
\frac{1}{q_1}=1-\frac{\alpha+2}{q}.
\end{cases}
\end{equation}
The first term in \eqref{a1} is finite provided that $\frac{N}{\gamma}>b$; but from the first equation in  \eqref{CLHs3} this is equivalent to $\alpha<\frac{(N-b)r-2N}{N-2r}$ (assuming $r<\frac{N}{2}$). Thus, taking into account our assumption $\alpha<4^*$, we see that is suffices to  choose, for instance, $r$ such that 
 $$
 \frac{(N-b)r-2N}{N-2r}=4^*.
 $$
Consequently, $r$ and $q$ are given by
\begin{equation}\label{APLHs}
 r=\frac{2N(N-b)}{N(N-4)-bN+16}\;\;\textnormal{and}\;\;q=\frac{2(N-b)}{N-4},
 \end{equation}
  where we have used that $(q,r)$ must be a $B$-admissible pair to compute the value of $q$. It is easy to see that with this choice we have $r<\frac{N}{2}$, since it is equivalent to $b<N-4$ (this is true because $N\geq 8$). In addition, from the second equation  in \eqref{CLHs3} and \eqref{APLHs} we also have
\begin{equation*}
\theta_2:=\frac{1}{q_1}=\frac{8-2b-\alpha(N-4)}{2(N-b)}>0,
\end{equation*}
because $\alpha<4^*$. Hence,
\begin{equation}\label{LHsB2} 
B_2 \leq c T^{\theta_2}\|\Delta u\|^\alpha_{B(L^2;I)} \|v\|_{B(L^2;I)}\leq cT^{\theta_2}\|u\|^\alpha_I \|v\|_I.
\end{equation}

 \ {\bf Case 1.2: $5\leq N\leq 7$}. Let us start by fixing the pair
 $$
 \left(q_\varepsilon,r_\varepsilon\right)=\left(\frac{8}{N-4-2\varepsilon},\frac{N}{2+\varepsilon}\right),
 $$
 where $\varepsilon>0$ is small it will be appropriately chosen later. Since $5\leq N\leq7$, it is easy to check that $(q_\varepsilon,r_\varepsilon)$ is  $B$-admissible, for any small $\varepsilon>0$. Now if $(q,r)$ is another $B$-admissible pair, Holder's inequality and Sobolev's embedding (note that $r_\varepsilon<\frac{N}{2}$) imply
 \begin{equation}\label{a2}
 \begin{split}
 B_2& \leq \left \|\||x|^{-b}\|_{L^\gamma(B)} \|u\|^\alpha_{L_x^{\alpha r_1}} \|v\|_{L^r_x}\right\|_{L^{q'}_I}\\
 &\leq \left \|\||x|^{-b}\|_{L^\gamma(B)} \|\Delta u\|^\alpha_{L_x^{r_\varepsilon}} \|v\|_{L^r_x}\right\|_{L^{q'}_I}\\
 &=  \||x|^{-b}\|_{L^\gamma(B)}T^{\frac{1}{q_1}} \|\Delta u\|^\alpha_{L_I^{q_\varepsilon}L^{r_\varepsilon}_x} \|v\|_{L^{q}_IL_x^r},
 \end{split}
 \end{equation}
 where
 \begin{equation}\label{CLHs21} 
 \begin{cases}
 \frac{1}{r'}=\frac{1}{\gamma}+\frac{1}{r_1}+\frac{1}{r},\\ 
 2=\frac{N}{r_\varepsilon}-\frac{N}{\alpha r_1}, \\ 
 \frac{1}{q'}=\frac{1}{q_1}+\frac{\alpha}{q_\varepsilon}+\frac{1}{q}.
 \end{cases}
 \end{equation}
 From \eqref{a2} we see that in order to complete the bound for $B_2$ it suffices to choose $(q,r)$ such that $\frac{N}{\gamma}>b$ and $\frac{1}{q_1}>0$. But, in view of \eqref{CLHs21},
 $$
 \frac{N}{\gamma}=N-\frac{2N}{r}-\alpha\varepsilon.
 $$
 Hence $\frac{N}{\gamma}>b$ is equivalent to $N-b>\frac{2N}{r}+\alpha\varepsilon$. Let us then choose $r$ given by
 $$
 r=\frac{2N}{N-b-2\alpha\varepsilon}.
 $$
By taking
 $
 q=\frac{8}{b+2\alpha\varepsilon}
 $
we now see that $(q,r)$ is $B$-admissible if $\varepsilon$ is chosen to be sufficiently small. In addition, since
$$
\frac{1}{q_1}=1-\frac{b+2\alpha\varepsilon}{4}-\frac{\alpha(N-4-2\varepsilon)}{8},
$$
we obtain that $\frac{1}{q_1}>0$ is equivalent to $\alpha<\frac{8-2b}{N-4+2\varepsilon}$, which certainly is true, for $\varepsilon$ sufficiently small, in view of our assumption $\alpha<4^*$.

\ {\bf Case 2: $N= 4$.} We use similar arguments as in the previous case. Let us start by estimating $B_2$. We have
\begin{equation}\label{LHsD21} 
B_2\leq \left\|\chi_B|x|^{-b}|u|^\alpha v\right\|_{L^{q'}_IL^{r'}_x} \leq T^{\frac{1}{q'}} \||x|^{-b}\|_{L^\gamma(B)} \|u\|^{\alpha}_{L^\infty_IL_x^{\alpha r_1}} \|v\|_{L^\infty_IL^2_x},
\end{equation}
where 
\begin{equation}\label{L<5}
\frac{1}{r'}=\frac{1}{\gamma}+\frac{1}{r_1}+\frac{1}{2}.
\end{equation}
Choose $r_1$ satisfying $\alpha r_1=\frac{2}{\delta}$, where $\delta \in (0,1)$ is sufficiently small.  Note that 
$$
\frac{4}{\gamma}=2-2\alpha\delta-\frac{4}{r}.
$$
Thus, for $\frac{4}{\gamma}>b$ it suffices to take\footnote{Observe that, since $b<2$ and $\delta$ is small we deduce that $2-b-2\delta \alpha>0$.} $r\in\left(\frac{4}{2-b-2\delta \alpha },+\infty\right)$.
 Therefore, from \eqref{LHsD21} and Sobolev's embedding,
\begin{equation}\label{L<53}
\begin{split}
B_2&\lesssim T^{\frac{1}{q'}}\|u\|^{\alpha}_{L^\infty_IH^{2-\frac{\delta}{2}}} \|v\|_{L^\infty_IL^2_x} 
\lesssim T^{\frac{1}{q'}}\|u\|^{\alpha}_{L^\infty_IH^{2}} \|v\|_{L^\infty_IL^2_x}
\lesssim T^{\theta_2}\|u\|^{\alpha}_I\|v\|_I,
\end{split}
\end{equation}
where we used that $(\infty,2)$ is $B$-admissible.

The idea to estimate $B_1$ is similar to that for $B_2$. Indeed, 
\begin{equation}\label{LHsD21a} 
B_1\ \leq T^{\frac{1}{q'}} \||x|^{-b}\|_{L^\gamma(B^C)} \|u\|^{\alpha}_{L^\infty_IL_x^{\alpha r_1}} \|v\|_{L^\infty_IL^2_x}.
\end{equation}
provided that \eqref{L<5} holds. With the same choice of $r_1$, we deduce that for $\frac{4}{\gamma}<b$ it suffices to choose  $r\in\left(2,\frac{4}{2-b-2\delta \alpha}\right)$, which implies $|x|^{-b}\in L^\gamma(B^C)$. Therefore,  the Sobolev embedding  implies 
\begin{equation*}
B_1\lesssim T^{\frac{1}{q'}} \|u\|^{\alpha}_{L_t^\infty H_x^{2}} \|v\|_{L_t^\infty L^2_x}\lesssim T^{\theta_1} \|u\|^{\alpha}_{I} \|v\|_{I}.
\end{equation*}

\ {\bf Case 3: $1\leq N\leq3$.}	The proof in this case is similar (and even easier) to that of Case 2, with the advantage that in view of Sobolev's embedding $L^\infty(\mathbb{R}^N)\hookrightarrow H^2(\mathbb{R}^N)$, we can take $r_1=\infty$. So we omit the details.
\end{proof}
\end{lemma}

\ The next lemma provides an estimate of the derivative of $F(x,u)$ in the norm of $L^{2}_IL_x^{\frac{2N}{N+2}}$, the dual space of $L^{2}_IL_x^{\frac{2N}{N-2}}$.

\begin{lemma}\label{LLH21} Let $N\geq 3$, $0< b<\min\{\frac{N}{2},4\}$  and $\max\left\{0,\frac{2-2b}{N}\right\}<\alpha<4^*$, then the following statement holds  
$$
\left \|\nabla(|x|^{-b}|u|^\alpha u)\right \|_{L^{2}_IL_x^{\frac{2N}{N+2}}}\leq c (T^{\theta_1}+T^{\theta_2})\| u\|^{\alpha+1}_{I},
$$
where $c,\theta_1,\theta_2 >0$.
\begin{proof} 
Observe that
\begin{equation*}
\left \|\nabla(|x|^{-b}|u|^\alpha u)\right \|_{L^{2}_IL_x^{\frac{2N}{N+2}}} \leq C_1+C_2, 
 \end{equation*}
where 
\begin{align*}
 \hspace{1.2cm} C_1= &\left \|\nabla(|x|^{-b}|u|^\alpha u)\right \|_{L^{2}_IL_x^{\frac{2N}{N+2}}(B^C)}\;\textnormal{and}\;\; C_2=\left \|\nabla(|x|^{-b}|u|^\alpha u)\right \|_{L^{2}_IL_x^{\frac{2N}{N+2}}(B)}.
 \end{align*}
Now we divide the proof according to the dimension $N\geq5$ and $N=3,4$.

{\bf Case 1: $N\geq 5$.} First we estimate $C_2$. Let $C_{22}(t)=\left\|\nabla(|x|^{-b}|u|^\alpha u)\right\|_{L_x^{\frac{2N}{N+2}}(B)}$.
Applying H\"older's inequality, we deduce 
\begin{equation}\label{L1C10}
\begin{split}
\hspace{0.5cm} C_{22}(t)& \leq  \||x|^{-b}\|_{L^\gamma(B)} \|\nabla(|u|^\alpha u) \|_{L^{\beta}_x} + \|\nabla(|x|^{-b})\|_{L^d(B)}\|u\|^{\alpha +1}_{L^{(\alpha+1)e}_x},   \\
&\lesssim   \| |x|^{-b} \|_{L^\gamma(B)}  \|u\|^\alpha_{L_x^{\alpha r_1}}   \| \nabla u \|_{L_x^{r_2}} + \||x|^{-b-1}\|_{L^d(B)}\| \Delta u\|^{\alpha +1}_{L^r_x},   \\
&\lesssim  \| |x|^{-b} \|_{L^\gamma(B)} \| \Delta u \|^{\alpha}_{L_x^{r}}\|\Delta u \|_{L_x^r} + \||x|^{-b-1}\|_{L^d(B)}\| \Delta u\|^{\alpha +1}_{L^r_x},\\
&\lesssim  \| |x|^{-b} \|_{L^\gamma(B)} \| \Delta u\|^{\alpha +1}_{L^r_x} + \||x|^{-b-1}\|_{L^d(B)}\| \Delta u\|^{\alpha +1}_{L^r_x},
\end{split}
\end{equation}
where we also have used the Sobolev inequality. Here, we must have the  relations
\[
\begin{cases}
\frac{N+2}{2N}=\frac{1}{\gamma}+\frac{1}{\beta}=\frac{1}{d}+\frac{1}{e},\\ 
\frac{1}{\beta}=\frac{1}{r_1}+\frac{1}{r_2}, \\
2=\frac{N}{r}-\frac{N}{\alpha r_1}=\frac{N}{r}-\frac{N}{(\alpha+1)e}\;,\;\;\;\;r<\frac{N}{2},\\ 
1=\frac{N}{r}-\frac{N}{r_2},
\end{cases}
\]
which in turn are equivalent to 
\begin{equation}\label{L1C22}
\begin{cases} 
\frac{N}{\gamma}=\frac{N}{2}-\frac{N(\alpha+1)}{r}+2\alpha +2,\\ 
\frac{N}{d}=\frac{N}{2}-\frac{N(\alpha+1)}{r}+2\alpha +3.
\end{cases}
\end{equation}
In order to obtain that $\||x|^{-b}\|_{L^\gamma(B)}$ and $\||x|^{-b-1}\|_{L^d(B)}$ are finite, we need $\frac{N}{\gamma}>b$ and $\frac{N}{d}>b+1$, respectively. But from \eqref{L1C22} we see that $\frac{N}{\gamma}>b$ if and only if  $\frac{N}{d}>b+1$. Hence, it is sufficient to check that
$$
\frac{N}{2}-\frac{N(\alpha+1)}{r}+2\alpha +2>b.
$$
This last inequality is equivalent to $\alpha < \frac{r(N+4-2b)-2N}{2(N-2r)}$; thus we can choose $r$ such that
$$
\frac{r(N+4-2b)-2N}{2(N-2r)}=\frac{8-2b}{N-4}.
$$
Therefore,
\begin{equation}\label{BA}
r=\frac{2N(N+4-2b)}{N^2-2bN+16}\;\;\;\textnormal{and}\;\;\;q=\frac{2(N+4-2b)}{N-4},
\end{equation}
and we can easily see that $(q,r)$ is $B$-admissible and $r<\frac{N}{2}$ (here we need to use that $b<\frac{N}{2}$).
Moreover, from \eqref{L1C10} we obtain
\begin{equation}\label{L1C11}
C_{22}(t)\lesssim \| \Delta u\|^{\alpha +1}_{L^r_x}.
\end{equation}
Finally, \eqref{L1C11} and the H\"older inequality in the time variable yield
\begin{eqnarray*}
C_2=\|C_{22}(t)\|_{L^2_I} 
\lesssim T^{\frac{1}{q_1}} \|\Delta u\|^{\alpha+1}_{L_I^{q}L^r_x}
\lesssim T^{\frac{1}{q_1}} \|u\|^{\alpha+1}_I,
\end{eqnarray*}  
where
\begin{equation*}\label{L1C222}
\frac{1}{2}=\frac{1}{q_1}+\frac{\alpha+1}{q}.
\end{equation*}
From \eqref{BA}, the last inequality and our assumption on $\alpha$, we conclude that  
$$
\frac{1}{q_1}=\frac{8-2b-\alpha(N-4)}{2(N+4-2b)}>0.
$$
Hence the estimate for $C_2$ follows with  $\theta_2=\frac{1}{q_1}$.


 We now estimate $C_1$. Here we need to divide the proof according to $b\geq2$ and $b<2$.\\
 
{\bf Subcase $b\geq 2$.} Let $C_{11}(t)=\left\|\nabla(|x|^{-b}|u|^\alpha u)\right\|_{L_x^{\frac{2N}{N+2}}(B^C)}$. Arguing as in the term $C_{22}(t)$, we have
\begin{equation}\label{Lb>=2}
C_{11}(t)  \lesssim  \| |x|^{-b} \|_{L^\gamma(B^C)}  \|u\|^\alpha_{L_x^{\alpha r_1}}   \| \nabla u \|_{L_x^{r_2}} + \| |x|^{-b-1} \|_{L^d(B^C)}  \|u\|^\alpha_{L_x^{\alpha e_1}}   \| u \|_{L_x^{e_2}}
\end{equation}
with
\begin{equation}\label{Lb>=21}
\frac{N}{\gamma}=\frac{N+2}{2}-\frac{N}{r_1}-\frac{N}{r_2}\;\;\textnormal{and}\;\;\frac{N}{d}=\frac{N+2}{2}-\frac{N}{e_1}-\frac{N}{e_2}.
\end{equation}
In view of Sobolev's inequality, we then deduce
\begin{equation}\label{Lb>=22}
C_{11}(t)\lesssim | |x|^{-b} \|_{L^\gamma(B)} \| \Delta u \|^{\alpha}_{L_x^{r}}\|\Delta u \|_{L_x^r} + \||x|^{-b-1}\|_{L^d(B)}\| \Delta u\|^{\alpha}_{L^r_x}\| \nabla u\|_{L^r_x},
\end{equation}
where, from \eqref{Lb>=21},
\begin{equation}\label{Lb>=23}
\frac{N}{\gamma}=\frac{N}{d}-1=\frac{N+4\alpha+4}{2}-\frac{N(\alpha+1)}{r}.\end{equation}
For $\varepsilon>0$ small, by choosing the $B$-admissible pair $(q,r)$  defined by
$$
q=\frac{8(\alpha+1)}{\alpha(N-4)-2\varepsilon}\;\;\textnormal{and}\;\;r=\frac{2N(\alpha+1)}{N+4\alpha+2\varepsilon},
$$
we deduce, after some calculations, that $\frac{N}{\gamma}-b=\frac{N}{d}-1-b<0$ (because $b\geq 2$), which implies that $\| |x|^{-b} \|_{L^\gamma(B^C)}$ and $\||x|^{-b-1}\|_{L^d(B^C)}$ are finite. Thus,  
\begin{equation}\label{Lb>=24}
 C_1=\|C_{11}(t)\|_{L^2}\lesssim T^{\frac{1}{q_1}}\left( \|\Delta u\|^{\alpha+1}_{L_I^{q}L^r_x}+\|\Delta u\|^{\alpha}_{L_I^{q}L^r_x}\|\nabla u\|_{L_I^{q}L^r_x}\right) \lesssim T^{\theta_1}\|u\|^{\alpha+1}_{I}, 
\end{equation}
where we used the interpolation inequality  $\| \nabla u \|_{L^q_IL_x^{r}}\lesssim \| u \|^{1-\theta}_{L^q_IL_x^{r}}\| \Delta u \|^\theta_{L^q_IL_x^{r}}$. Here, $$
\theta_1=\frac{1}{q_1}=\frac{1}{2}-\frac{\alpha+1}{q}=\frac{4-\alpha(N-4)+2 \varepsilon}{8},$$ which is positive because $\alpha<\frac{8-2b}{N-4}$ and $b\geq 2$. \\

{\bf Subcase $b<2$.} The procedure is similar to that above. However, we need to divide the proof into five cases, because according to the range of the nonlinearity we need to choose different admissible pairs.

\ {\bf Case $A$: $\frac{4}{N-4}\leq \alpha<\frac{8-2b}{N-4}$.} Here, we choose the $B$-admissible pair $(q,r)=\left(\frac{8(\alpha+1)}{\alpha(N-4)-4},\frac{2N(\alpha+1)}{N+4\alpha+4}\right)$. The restriction on $\alpha$  ensures that $r\in \left[2,\frac{2N}{N-4}\right)$. Hence, by \eqref{Lb>=23} we obtain that $\frac{N}{\gamma}=\frac{N}{d}-1<b$. Therefore, from  \eqref{Lb>=22} and H\"older's inequality (because $\frac{1}{q_1}=\frac{1}{2}-\frac{\alpha+1}{q}=\frac{8-\alpha(N-4)}{8}$) we obtain \eqref{Lb>=24}.

\ {\bf Case $B$: $\frac{2}{N-4}\leq \alpha<\frac{4}{N-4}$.} Setting the $B$-admissible pair  $(q,r)=\left(\frac{8}{\alpha(N-4)-2},\frac{2N}{N+2-\alpha(N-4)}\right)$ and choosing $\alpha r_1=\alpha e_1=\frac{2N}{N-4}$, $r_2=e_2=r$ in \eqref{Lb>=21}, it follows that $\frac{N}{\gamma}-b<0$ and $\frac{N}{d}-b-1<0$.  Note also that since $\alpha\in \left[ \frac{2}{N-4},\frac{4}{N-4} \right)$ we have  $2\leq r<  \frac{2N}{N-4}$. Thus, \eqref{Lb>=2} yields
\begin{equation}\label{C11}
C_{11}(t)\lesssim \| u \|^{\alpha}_{L_x^{\frac{2N}{N-4}}}\left ( \|\nabla u \|_{L_x^r} + \| u \|_{L_x^r} \right). \end{equation}
 By using Sobolev's embedding and H\"older's inequality one has
\begin{equation}\label{C_1}
  C_1\lesssim  \| u \|^{\alpha}_{L^{\infty}_IL_x^{\frac{2N}{N-4}}}T^{\frac{1}{q_1}}\left ( \|\nabla u \|_{L^q_IL_x^r} + \| u \|_{L^q_IL_x^r}\right)\lesssim T^{\theta_1}\| u \|^{\alpha}_{L^{\infty}_IH_x^{2}}\|u\|_I\lesssim T^{\theta_1}\|u\|^{\alpha+1}_I.  
\end{equation}
Note that here we also have $\theta_1=\frac{1}{q_1}=\frac{1}{2}-\frac{1}{q}>0$. 

\ {\bf Case $C$: $ \frac{2}{N-2}\leq  \alpha<\frac{2}{N-4}$.} Taking into account the range of $\alpha$, we can choose $(q,r)=\left (\frac{8}{\alpha(N-2)-2},\frac{2N}{N+2-\alpha(N-2)}\right)$, $\alpha r_1=\alpha e_1=\frac{2N}{N-2}$, $r_2=e_2=r$.  Arguing as in Case B, we obtain \eqref{C_1}.

 \ {\bf Case $D$: $ \frac{2}{N}\leq  \alpha<\frac{2}{N-2}$.} In this case it suffices to choose $(q,r)=(\frac{8}{N\alpha-2},\frac{2N}{N+2-N\alpha})$ and $\alpha r_1=\alpha e_1=2$, $r_2=e_2=r$, and proceed as above. 
 
 \ {\bf Case $E$: $\max\left\{0,\frac{2(1-b)}{N}\right\}<  \alpha<\frac{2}{N}$.}
 Finally, if $\alpha r_1=\alpha e_2=2$ and $r_2=e_2=2$, then $\frac{N}{\gamma}=\frac{N}{d}=1-\frac{N\alpha}{2}$. Thus, since $\alpha \in \left(\max\left\{0,\frac{2(1-b)}{N}\right\},\frac{2}{N}\right)$  we deduce $\frac{N}{\gamma}-b<0$, which also implies $\frac{N}{d}-b-1<0$. Therefore, from \eqref{Lb>=2} and \eqref{Lb>=24}, $$C_1\lesssim T^{\frac{1}{2}}\|u\|^{\alpha+1}_{L^\infty_IH^2_x}\lesssim T^{\frac{1}{2}}\|u\|^{\alpha+1}_I.$$ 
 
Note that Cases $A$-$E$ cover the range $\left(\max\left\{0,\frac{2(1-b)}{N}\right\},\frac{8-2b}{N-4}\right)$ for the parameter $\alpha$. Hence, the proof of the lemma is completed in the case  $N\geq 5$.

\

{\bf Case 2: $N=3,4$.} Following the notation in Case 1, we split integration in space on $B$ and $B^C$, obtaining terms $C_1$ and $C_2$. We start by estimating  $C_2$. If $(q,r)$ is any $B$-admissible pair, H\"older's inequality implies
\begin{equation}\label{L<54} 
\begin{split}
C_2 &\leq  \left \| \| |x|^{-b} \|_{L^\gamma(B)}  \|u\|^\alpha_{L_x^{\alpha r_1}}   \| \nabla u \|_{L_x^{r}}\right\|_{L^2_I} + \left \|\||x|^{-b-1}\|_{L^d(B)}\| u\|^{\alpha}_{L^{\alpha e}_x}\|u\|_{L^r_x}\right\|_{L^2_I}\\
&\lesssim   \| |x|^{-b} \|_{L^\gamma(B)}\|u\|^\alpha_{L^\infty_IL_x^{\alpha r_1}}T^{\frac{1}{q_1}}  \| \nabla u \|_{L^q_IL_x^{r}} + \||x|^{-b-1}\|_{L^d(B)}\|u\|^{\alpha}_{L^\infty_IL^{\alpha e}_x}T^{\frac{1}{q_1}}\|u\|_{L^q_IL^r_x}\\
&\lesssim T^{\frac{1}{q_1}}\left( \| |x|^{-b} \|_{L^\gamma(B)}\|u\|^\alpha_{L^\infty_IL_x^{\alpha r_1}}  \| u \|_{I} + \||x|^{-b-1}\|_{L^d(B)}\|u\|^{\alpha}_{L^\infty_IL^{\alpha e}_x}\|u\|_{I}\right),
\end{split}
\end{equation}
where 
\begin{equation}\label{n12}
\frac{N}{\gamma}-b=\frac{N+2-2b}{2}-\frac{N}{r_1}-\frac{N}{r}, \qquad\frac{N}{d}-b-1=\frac{N-2b}{2}-\frac{N}{e}-\frac{N}{r},
\end{equation}
and
 $$\frac{1}{2}=\frac{1}{q_1}+\frac{1}{q}.$$ First we choose $\alpha r_1=\alpha e=\frac{2}{\delta}$ with $\delta \in (0,1)$ if $N=4$ and $\delta=0$ if $N=3$. Next, for $\varepsilon>0$ small we choose $(q,r)=\left(\frac{8}{N(1-\varepsilon)},\frac{2}{\varepsilon}\right)$, which is $B$-admissible in dimension $N\leq4$. With these choices in hand,
$$
\frac{N}{\gamma}-b=\frac{N+2-2b-\delta \alpha N-\varepsilon N}{2}\;\;\;\textnormal{and}\;\;\;\frac{N}{d}-b-1=\frac{N-2b-\delta \alpha N -\varepsilon N}{2},
$$
which are positive in view of our assumption $b<\frac{N}{2}$. Therefore, $|x|^{-b}\in L^\gamma(B)$ and $|x|^{-b-1}\in L^d(B)$. Furthermore, from \eqref{L<54}, noting that since $\alpha r_1>2$,  $\alpha e>2$ and using the Sobolev embedding,  we obtain 
\begin{equation*}
C_2 \lesssim   T^{\theta_2} \|u\|^\alpha_{L^\infty_IH_x^{2}}  \|  u \|_{I}\lesssim T^{\theta_2} \|u\|^{\alpha+1}_{I},
\end{equation*}
where $\theta_2=\frac{1}{q_1}>0$, taking into account that $q>2$.

Now we estimate $C_1$. Indeed, repeating the same argument to obtain \eqref{L<54} and choosing $\alpha r_1=2=\alpha e$ we get
\begin{equation*}
C_1\lesssim T^{\frac{1}{q_1}}\left( \| |x|^{-b} \|_{L^\gamma(B^C)}\|u\|^\alpha_{L^\infty_IL_x^{2}}  \| u \|_{I} + \||x|^{-b-1}\|_{L^d(B^C)}\|u\|^{\alpha}_{L^\infty_IL^{2}_x}\|u\|_{I}\right),
\end{equation*}
and 
$$
\frac{N}{\gamma}-b=\frac{N+2-2b-\alpha N}{2}-\frac{N}{r}\;\;\;\textnormal{and}\;\;\;\frac{N}{\gamma}-b-1=\frac{N-2b-\alpha N}{2}-\frac{N}{r}.
$$
To obtain $\frac{N}{\gamma}-b<0$ and $\frac{N}{\gamma}-b-1<0$ it  suffices  to  take the $B$-admissible pair such that  $r\in\left (2,\frac{2N}{N+2-2b-\alpha N}\right)$. Note that our assumption $\alpha >\frac{2(1-b)}{N}$ gives $\frac{2N}{N+2-2b-\alpha N}>2$. Hence, $\||x|^{-b}\|_{L^\gamma(B^C)}$ and $\||x|^{-b-1}\|_{L^d(B^C)}$ are finite and 
\begin{equation*}
C_1 \lesssim T^{\theta_1} \|u\|^\alpha_{L^\infty_IL_x^{2}}  \|  u \|_{I}\lesssim T^{\theta_1} \|u\|^{\alpha+1}_{I},
\end{equation*}
with   $\theta_1=\frac{1}{q_1}=\frac{1}{2}-\frac{1}{q}>0$.

The proof of the lemma is then completed.
\end{proof}
\end{lemma}



\ We now have all tools to prove the main result of this section, Theorem \ref{LWPH2}. 

\begin{proof}[\bf{Proof of Theorem \ref{LWPH2}}]	
We use the contraction mapping principle again. To do so, we define $$
X= C\left([-T,T];H^2(\mathbb{R}^N)\right)\bigcap L^q\left([-T,T];H^{2,r}(\mathbb{R}^N)\right),$$ where $(q,r)$ is any $B$-admissible pair and $T>0$ will be determined properly later. Also, in $X$ we define the norm
\begin{equation*}\label{NHs} 
\|u\|_T=\|u\|_{B\left(L^2;[-T,T]\right)}+\|\Delta u\|_{B\left(L^2;[-T,T]\right)}.
\end{equation*}
We shall show that the mapping $G$ defined in \eqref{OPERATOR} is a contraction on the complete metric space
\begin{equation*}
S(a,T)=\{u \in X : \|u\|_T\leq a \}
\end{equation*}
with the metric 
$$
d_T(u,v)=\|u-v\|_{B\left(L^2;[-T,T]\right)},
$$
for a suitable choice of the parameters $a$ and $T$.

\ As in the proof of Theorem \ref{LWPL2}, without loss of generality we consider only the case $t > 0$. Note that, in particular, we have $\|\cdot\|_T=\|\cdot\|_I$. Let us first show that $G$ is well defined from $S(a,T)$ to $S(a,T)$. Indeed, if $F(x,u)=|x|^{-b}|u|^\alpha u$, following the same arguments as in the proof of Theorem \ref{LWPL2}, we obtain
\begin{equation}\label{NCD0} 
\begin{split} 
\|G(u)\|_{B\left(L^2;I\right)}&\leq c\|u_0\|_{L^2}+ \left \|\chi_{B^C}F(x,u) \right\|_{B'(L^2;I)}+\left \|\chi_{B}F(x,u) \right\|_{B'(L^2;I)}\\
&\leq  c\|u_0\|_{L^2}+ c(T^{\theta_1}+T^{\theta_2}) \| u \|^{\alpha+1}_T,
\end{split}
\end{equation}
where in the last inequality we used Lemma \ref{LLH210}. Also, from \eqref{EstimativaImportante} and Lemma \ref{LLH21},
\begin{equation}\label{NCD}  
\begin{split}
\|\Delta G(u)\|_{B\left(L^2;I\right)}&\leq c \|\Delta u_0\|_{L^2}+ c \|F(x,u)\|_{L^2_IL_x^{\frac{2N}{N+2}}}\\
&\leq c \|\Delta u_0\|_{L^2}+c(T^{\theta_1}+T^{\theta_2}) \| u \|^{\alpha+1}_T.
\end{split}
\end{equation}
By combining \eqref{NCD0} and \eqref{NCD}, we see that if $u \in S(a,T)$, then
\begin{equation*}\label{NI}
\|G(u)\|_T\leq c \|u_0\|_{H^2}+c (T^{\theta_1}+T^{\theta_2}) a^{\alpha+1}.
\end{equation*}
Consequently, by choosing $a=2c\|u_0\|_{H^2}$ and $T>0$ such that 
\begin{equation}\label{CTHs} 
c a^{\alpha}(T^{\theta_1}+T^{\theta_2}) < \frac{1}{4},
\end{equation}
we obtain $G(u)\in S(a,T)$. Hence, $G$ is well defined on $S(a,T)$.

\ To prove that $G$ is a contraction on $S(a,T)$ with respect to the metric $d_T$ we use \eqref{FEI} and Lemma \ref{LLH210} to deduce
\[
\begin{split}
d_T(G(u),G(v))&\leq   c\left\| \chi_{B^C}\left( F(x,u)-F(x,v)\right) \right\|_{B'\left(L^2;I\right)}
+\left\| \chi_B\left( F(x,u)-F(x,v)\right) \right\|_{B'\left(L^2;I\right)}\\ 
&\leq  c (T^{\theta_1}+T^{\theta_2})\left(\|u\|^\alpha_T+\|v\|^\alpha_T\right)d_T(u,v),
\end{split}
\]
So, for $u,v\in S(a,T)$, we get
$$
 d_T(G(u),G(v))\leq c (T^{\theta_1}+T^{\theta_2})a^\alpha d_T(u,v).
$$
Therefore, from \eqref{CTHs}, $G$ is a contraction on $S(a,T)$ and by the contraction mapping principle we have a unique fixed point $u \in S(a,T)$ of $G$. This completes the proof of the theorem.
\end{proof}
		

\section{\bf Global Well-Posedness and Scattering}\label{secglo}
	
\ The goal of this section is to study the global well-posedness of the Cauchy problem \eqref{IBNLS}. 

\subsection{Global Well-Posedness in $L^2$}
The global well-posedness result in $L^2(\mathbb{R}^N)$ (Theorem \ref{GWPL2}) is an immediate consequence of Theorem \ref{LWPL2}. Indeed, using \eqref{NT3} we obtain that the existence time depends only on the $L^2$ norm of the initial data, that is, $ T:=T(\|u_0\|_{L^2})= \frac{c}{\|u_0\|_{L^2}^d}$, for some constants $c,d>0$. Hence, the conservation law \eqref{mass} allows us to reapply Theorem \ref{LWPL2} as many times as we wish preserving the length of the time interval. This gives us the global solution.

\subsection{Global Well-Posedness in $H^2$}
In this subsection, we turn our attention to prove Proposition \ref{glosub} and Theorem \ref{GWPH2}. Before proving Proposition \ref{glosub} we recall the following version of the Caffarelli-Kohn-Nirenberg-type inequality.

\begin{lemma}\label{k3lem}
Assume  $N\geq1$, $0<b<\min\left\{N,4\right\}$, and $0<\alpha<4^*$. Then
$$
\int_{\mathbb{R}^N}|x|^{-b}|u|^{\alpha+2}dx\lesssim  \|\Delta u\|_{L^2}^{\frac{N\alpha+2b}{4}}\| u\|_{L^2}^{\alpha+2-\frac{N\alpha+2b}{4}}.
$$
\end{lemma}
\begin{proof}
See \cite[page 1516]{Lin}.
\end{proof}

\begin{proof}[{\bf Proof of Proposition \ref{glosub}}]
As in the case of $L^2$-solutions, the existence time obtained in Theorem \ref{LWPH2} depends only the $H^2$ norm of the initial data. Hence, to obtain a global solution it is sufficient to get an a priori bound of the local solution. To do so, from Lemma  \ref{k3lem}, the conservation of the mass and the energy, we obtain
	\[
	\begin{split}
	\|\Delta u(t)\|_{L^2}^2&=2E[u_0]+c\int|x|^{-b}|u(t)|^{\alpha+2}dx\\
	&\leq 2E[u_0]+c\|\Delta u(t)\|_{L^2}^{\frac{\alpha N+2b}{4}} \|u_0\|_{L^2}^{\alpha+2-\frac{\alpha N+2b}{4}}
	\end{split}
	\]
	If $\alpha<\frac{8-2b}{N}$ the above inequality promptly implies that the Laplacian of $u$ remains bounded, as long as the local solution exists, which in turn implies the global existence. On the other hand, if $\alpha=\frac{8-2b}{N}$ we deduce that
	$$
	(1-c\|u_0\|_{L^2}^\alpha)\|\Delta u(t)\|_{L^2}^2\leq 2E[u_0].
	$$
	Hence, the Laplacian of $u$ remains bounded if $\|u_0\|_{L^2}\leq c^{-1/\alpha}$, which completes the proof of the proposition.
\end{proof}

 Next we turn attention to the proof of Theorem \ref{GWPH2}.  Its core  is to establish suitable estimates on the nonlinearity $F(x,u)=|x|^{-b}|u|^\alpha u$.

\begin{lemma}\label{lemmaglobal1} 
Let $N\geq 3$ and $0<b<\min\{\frac{N}{2},4\}$. If $ \frac{8-2b}{N}<\alpha<4^*$ then the following statements hold:
\begin{itemize}
\item [(i)] $\left \|\chi_B|x|^{-b}|u|^\alpha v \right\|_{B'(\dot{H}^{-s_c})}+
\left \|\chi_{B^C}|x|^{-b}|u|^\alpha v \right\|_{B'(\dot{H}^{-s_c})} \leq c\| u\|^{\theta}_{L^\infty_tH^2_x}\|u\|^{\alpha-\theta}_{B(\dot{H}^{s_c})} \|v\|_{B(\dot{H}^{s_c})};
$
\item [(ii)] $\left\|\chi_B|x|^{-b}|u|^\alpha v \right\|_{B'(L^2)} + \left\|\chi_{B^C}|x|^{-b}|u|^\alpha v \right\|_{B'(L^2)} 
\leq c\| u\|^{\theta}_{L^\infty_tH^2_x}\|u\|^{\alpha-\theta}_{B(\dot{H}^{s_c})} \| v\|_{B(L^2)}$,
\end{itemize} 
where $c>0$ and $\theta\in (0,\alpha)$ is a sufficiently small number.	
\begin{proof} Before starting the proof we define following numbers
	\begin{equation}\label{PHsA11}  
\widehat{q}=\frac{8\alpha(\alpha+2-\theta)}{\alpha(N\alpha+2b)-\theta(N\alpha-8+2b)}\;\;\;\widehat{r}\;=\;\frac{N\alpha(\alpha+2-\theta)}{\alpha(N-b)-\theta(4-b)}
\end{equation}
and 
\begin{align}\label{PHsA22}
\widetilde{a}\;=\;\frac{4\alpha(\alpha+2-\theta)}{\alpha[N(\alpha+1-\theta)-4+2b]-(8-2b)(1-\theta)}\;\;\;  \widehat{a}=\frac{4\alpha(\alpha+2-\theta)}{8-2b-(N-4)\alpha},
\end{align}
 It is easily seen that, for $\theta$ sufficiently small, $(\widehat{q},\widehat{r})\in \mathcal{B}_0$, $(\widehat{a},\widehat{r})\in \mathcal{B}_{s_c}$,  $(\widetilde{a},\widehat{r})\in \mathcal{B}_{-s_c}$, and
 \begin{equation}\label{PHsA22.1}
\frac{1}{\widetilde{a}'}=\frac{\alpha-\theta}{\widehat{a}}+\frac{1}{\widehat{a}}, \quad \mbox{and}\quad  \frac{1}{\widehat{q}'}=\frac{\alpha-\theta}{\widehat{a}}+\frac{1}{\widehat{q}}.
 \end{equation}

Let us prove (i). To do so, let $A\subset \mathbb{R}^N$ denote either $B$ or $B^C$. It is then sufficient to estimate $\left \|\chi_A|x|^{-b}|u|^\alpha v \right\|_{B'(\dot{H}^{-s_c})}$. Clearly, we have
 $\left \|\chi_A|x|^{-b}|u|^\alpha v \right\|_{B'(\dot{H}^{-s_c})}\leq \left \|\chi_A|x|^{-b}|u|^\alpha v \right\|_{L^{\widetilde{a}'}_tL^{\widehat{r}'}_x}$. But, from  H\"older's inequality one has
 \begin{equation}\label{LG1Hs1}
\begin{split}
\left \| \chi_A |x|^{-b}|u|^\alpha v \right\|_{L^{\widehat{r}'}_x} &\leq \||x|^{-b}\|_{L^\gamma(A)}  \|u\|^{\theta}_{L^{\theta r_1}_x}   \|u\|^{\alpha-\theta}_{L_x^{(\alpha-\theta)r_2}}  \|v\|_{L^{\widehat{r}}_x}  \\
&=\||x|^{-b}\|_{L^\gamma(A)}  \|u\|^{\theta}_{L^{\theta r_1}_x}   \|u\|^{\alpha-\theta}_{L_x^{\widehat{r}}} \|v\|_{L^{\widehat{r}}_x},
\end{split}
\end{equation}
where
\begin{equation}\label{LG1Hs2}
\frac{1}{\widehat{r}'}=\frac{1}{\gamma}+\frac{1}{r_1}+\frac{1}{r_2}+\frac{1}{\widehat{r}}\;\;\textnormal{and}\;\;\widehat{r}=(\alpha-\theta)r_2.
\end{equation}
Observe that \eqref{LG1Hs2} implies
\begin{equation*}
	\frac{N}{\gamma}=N-\frac{N(\alpha+2-\theta)}{\widehat{r}}-\frac{N}{r_1},
\end{equation*}
and from \eqref{PHsA11} it follows that
\begin{equation}\label{LG1Hs3}
\frac{N}{\gamma}-b=\frac{\theta(4-b)}{\alpha}-\frac{N}{r_1}.
\end{equation}

Now, we make use of the Sobolev embedding (Lemma \ref{SI}), so we consider three cases: $N\geq 5$, $N=4$ and $N=3$.

\ {\bf Case $N\geq 5$.} 
If $A=B$ we  choose $\theta r_1=\frac{2N}{N-4}$, so that $\frac{N}{\gamma}-b=\theta(2-s_c)>0$ (recall that $s_c<2$). On the other hand, if $A=B^C$ we choose $\theta r_1=2$, so that $\frac{N}{\gamma}-b=-\theta s_c<0$. Thus, in both cases the quantity $\||x|^{-b}\|_{L^\gamma(A)}$ is finite and, by Sobolev embedding, $H^2\hookrightarrow L^{\theta r_1}$. Therefore, from \eqref{LG1Hs1},
\begin{equation}\label{LG1Hs41}
\left \| \chi_A |x|^{-b}|u|^\alpha v \right\|_{L^{\widehat{r}'}_x} \lesssim \|u\|^{\theta}_{H^2_x}  \|u\|^{\alpha-\theta}_{L_x^{\widehat{r}}} \|v\|_{L^{\widehat{r}}_x}.
\end{equation}
An application of H\"older's inequality in time, taking into account \eqref{PHsA22.1}, now gives
\begin{equation}
\begin{split}\label{LG1Hs42}
\left \|\chi_A  |x|^{-b}|u|^\alpha v \right\|_{L_t^{\widetilde{a}'}L^{\widehat{r}'}_x}
&\lesssim \|u\|^{\theta}_{L^\infty_tH^2_x} \|u\|^{\alpha-\theta}_{ L_t^{\widehat{a}} L_x^{\widehat{r}}} \|v\|_{L^{\widehat{a}}_tL^{\widehat{r}}_x}   \\
&\lesssim \| u\|^{\theta}_{L^\infty_tH^2_x}\|u\|^{\alpha-\theta}_{B(\dot{H}^{s_c})} \|v\|_{B(\dot{H}^{s_c})},
\end{split}
\end{equation}
which completes the proof in this case.

\ {\bf Case $N=4$.} Following the same strategy as in Case $N\geq5$, it suffices to choose  $\theta r_1$ such that $\||x|^{-b}\|_{L^\gamma(A)}$ is finite and  $H^2\hookrightarrow L^{\theta r_1}$. Since $\alpha>\frac{8-2b}{N}$ we deduce that $\frac{N\alpha}{4-b}>2$. Thus, if $A=B$, by choosing
\begin{equation*}
\theta r_1\in \left(\frac{N\alpha}{4-b},+\infty\right),
\end{equation*}
from \eqref{LG1Hs3}, we immediately get $\frac{N}{\gamma}-b>0$. Furthermore,   if $A=B^C$, by choosing 
\begin{equation*}
\theta r_1\in\left(2,\frac{N\alpha}{4-b}\right).
\end{equation*}
we get $\frac{N}{\gamma}-b<0$. Again, in both cases we have $\||x|^{-b}\|_{L^\gamma(A)}<+\infty$ and $H^2\hookrightarrow L^{\theta r_1}$ (recall that, for $N=4$,  one has $H^2\hookrightarrow L^p,$  $p\in [2,\infty)$). 

\ {\bf Case $N=3$.} Here, recalling that $H^2\hookrightarrow L^\infty$ it suffices to take $r_1=\infty$, if $A=B$ and $\theta r_1=2$, if $A=B^C$. In the first case, we get  $\frac{N}{\gamma}-b=\frac{\theta(4-b)}{\alpha}>0$ and in the second one, $\frac{N}{\gamma}-b=-\theta s_c<0$. This completes the proof of part (i)

Since $(\widehat{q},\widehat{r})$ is $B$-admissible, the proof of (ii)  runs as in (i). We only point out that, once we obtain \eqref{LG1Hs41}, in view of \eqref{PHsA22.1},
 \begin{equation}\label{LGHsii}
\left\| \chi_A |x|^{-b}|u|^\alpha v\right \|_{L_t^{\widehat{q}'}L^{\widehat{r}'}_x}\leq  c \|u\|^{\theta}_{L^\infty_tH^2_x}\|u\|^{\alpha-\theta}_{L_t^{\widehat{a}} L_x^{\widehat{r}}} \|v\|_{L^{\widehat{q}}_tL^{\widehat{r}}_x},
\end{equation}
which yields (ii).
\end{proof}
\end{lemma}

\begin{remark}\label{RGP} As an immediate consequence of Lemma \ref{lemmaglobal1} (ii), if $A$ denotes either $B$ or $B^C$, we obtain
the following estimate, for $\alpha>\max\left\{1,\frac{8-2b}{N}\right\}$,
$$
\left \|\chi_A|x|^{-b}|u|^{\alpha-1} v  w\right\|_{B'(L^2)} \leq c \| u\|^{\theta}_{L^\infty_tH^2_x}\|u\|^{\alpha-1-\theta}_{B(\dot{H}^{s_c})} \|v\|_{B(\dot{H}^{s_c})}\|  w\|_{B(L^2)},
$$
where $\theta\in(0,\alpha-1)$ is a sufficiently small number. Indeed, we can repeat all the computations above, by replacing $|u|^\alpha v=|u|^\theta |u|^{\alpha-\theta }v$ by $|u|^{\alpha-1} vw=|u|^\theta |u|^{\alpha-1-\theta }vw$. This will  be used in the stability theory below. 
\end{remark}

\begin{lemma}\label{lemmaglobal2} 
Assume $N\geq 8$, $0<b<4$ and $ \frac{8-2b}{N}<\alpha<\frac{8-2b}{N-4}$. Then,
\begin{equation}\label{ELG3} 
\left\|\nabla \left(|x|^{-b}|u|^\alpha u\right) \right\|_{L_t^2L_x^{\frac{2N}{N+2}}}\leq c\| u\|^{\theta}_{L^\infty_tH^2_x}\|u\|^{\alpha-\theta}_{B(\dot{H}^{s_c})} \| \Delta u\|_{B(L^2)},
\end{equation}
where $c>0$ and $\theta\in (0,\alpha)$ is a sufficiently small number.
\end{lemma}	      
\begin{proof} Let
\begin{align}\label{pairN>=8a}
a\;=\;\frac{8\alpha(\alpha+1-\theta)}{8-2b-\alpha(N-4)}\;,\;\;\;\;r=\frac{2\alpha N(\alpha+1-\theta)}{\alpha(N+4-2b)-2\theta (4-b)}
\end{align}
and
\begin{align}\label{pairN>=8b}
q\;=\;\frac{8\alpha(\alpha+1-\theta)}{\alpha(N\alpha-4+2b)-\theta(N\alpha -8+2b)}.
\end{align}
Since $\theta>0$ is small it follows easily that $(q,r)$ is $B$-admissible and $(a,r)$ is $\dot{H}^{s_c}$-biharmomic admissible. Also, since  $N\geq 8$ we have $r<\frac{N}{2}$. In addition,
\begin{equation}\label{holglo}
\frac{1}{2}=\frac{\alpha-\theta}{a}+\frac{1}{q}.
\end{equation}

 Let $E(t)=\left\|\nabla(|x|^{-b}|u|^\alpha u)\right\|_{L_x^{\frac{2N}{N+2}}(A)}$, where $A$ denotes either $B$ or $B^C$. It follows from H\"older's inequality and Sobolev embedding  that
\begin{equation}\label{L>=8}
\begin{split}
 E(t) &\lesssim  \||x|^{-b}\|_{L^\gamma(A)} \|\nabla(|u|^\alpha u) \|_{L^{\beta}_x} + \||x|^{-b-1}\|_{L^d(A)}\|u\|^\theta_{L_x^{\theta p_1}}\|u\|^{\alpha-\theta}_{L^{(\alpha-\theta)p_2}_x} \|u\|_{L_x^{p_3}} \nonumber  \\
&\lesssim \||x|^{-b}\|_{L^\gamma(A)}  \|u\|^\theta_{L_x^{\theta r_1}}  \|u\|^{\alpha-\theta}_{L_x^{(\alpha-\theta)r_2}}   \| \nabla u \|_{L_x^{r_3}} + \||x|^{-b-1}\|_{L^d(A)} \|u\|^\theta_{L_x^{\theta p_1}}\|u\|^{\alpha-\theta}_{L^{r}_x} \|\Delta u\|_{L_x^{r}}\nonumber  \\
&\lesssim \||x|^{-b}\|_{L^\gamma(A)} \|u\|^\theta_{L_x^{\theta r_1}}  \|u\|^{\alpha-\theta}_{L_x^{r}}   \| \Delta u \|_{L_x^{r}} + \||x|^{-b-1}\|_{L^d(A)} \|u\|^\theta_{L_x^{\theta p_1}}\|u\|^{\alpha-\theta}_{L^{r}_x} \|\Delta u\|_{L_x^{r}}, 
\end{split}
\end{equation}
where
\[
\begin{cases}
\frac{N+2}{2N}=\frac{1}{\gamma}+\frac{1}{\beta}=\frac{1}{d}+\frac{1}{p_1}+\frac{1}{p_2}+\frac{1}{p_3},\\ 
\frac{1}{\beta}=\frac{1}{r_1}+\frac{1}{r_2}+\frac{1}{r_3}\\ 
(\alpha-\theta)r_2=(\alpha-\theta)p_2=r\\ 
1=\frac{N}{r}-\frac{N}{r_3}\;,\;\;\;\;2=\frac{N}{r}-\frac{N}{p_3}.
\end{cases}
\]
But from the definition of $r$ in \eqref{pairN>=8a} we deduce
\begin{equation}
\begin{cases}\label{L2E} 
\frac{N}{\gamma}-b=\frac{N}{2}+2-b-\frac{N}{r_1}-\frac{N(\alpha+1-\theta)}{r}=\frac{\theta(4-b)}{\alpha}-\frac{N}{r_1}\\ 
\frac{N}{d}-b-1=\frac{N}{2}+2-b-\frac{N}{p_1}-\frac{N(\alpha+1-\theta)}{r}=\frac{\theta(4-b)}{\alpha}-\frac{N}{p_1}.
\end{cases}
\end{equation}
Notice that the right hand side of \eqref{L2E} is the same as in \eqref{LG1Hs3}. Thus, as in the proof of Lemma \ref{lemmaglobal1}, by choosing $\theta r_1=\theta p_1=2$ if $A=B^C$ and $\theta r_1=\theta p_1=\frac{2N}{N-4}$ if $A=B$, we obtain  
$$
E(t)\lesssim \|u\|^\theta_{H_x^{2}}  \|u\|^{\alpha-\theta}_{L_x^{r}}\| \Delta u \|_{L_x^{r}},
$$
which implies 
$$
\left\|\nabla(|x|^{-b}|u|^\alpha u)\right\|_{L_x^{\frac{2N}{N+2}}}\lesssim \|u\|^\theta_{H_x^{2}}  \|u\|^{\alpha-\theta}_{L_x^{r}}\| \Delta u \|_{L_x^{r}}.
$$
Finally, in view of \eqref{holglo}, H\"older's inequality implies 
\[
\begin{split}
\left\|\nabla \left(|x|^{-b}|u|^\alpha u\right) \right\|_{L_t^2L_x^{\frac{2N}{N+2}}} &\lesssim \|u\|^{\theta}_{L^\infty_tH^2_x} \|u\|^{\alpha-\theta}_{L^{a}_tL^{r}_x}\|\Delta u\|_{L^{q}_tL^{r}_x} \\
 &\lesssim \| u\|^{\theta}_{L^\infty_tH^2_x}\|u\|^{\alpha-\theta}_{B(\dot{H}^{s_c})} \| \Delta u\|_{B(L^2)},
\end{split}
\]
which is the desired conclusion.
\end{proof}

Note that in the proof Lemma \ref{lemmaglobal2},  the condition $r<\frac{N}{2}$ is not valid for $N=5,6,7$.  So that we cannot apply the Sobolev embedding. In this case, we have the following.

\begin{lemma}\label{lemmaparcial}
	Let $N=5,6,7$ and $0<b<\frac{N^2-8N+32}{8}$. If $ \frac{8-2b}{N}<\alpha<\frac{N-2b}{N-4}$, then
	$$
	\left\|\nabla \left(|x|^{-b}|u|^\alpha u\right) \right\|_{L_t^2L_x^{\frac{2N}{N+2}}}\leq c\| u\|^{\theta}_{L^\infty_tH^2_x}\|u\|^{\alpha-\theta}_{B(\dot{H}^{s_c})} \| \Delta u\|_{B(L^2)},
	$$
	where $c>0$ and $\theta\in (0,\alpha)$ is a sufficiently small number.
\end{lemma} 
\begin{proof}
 The proof is similar to that of Lemma \ref{lemmaglobal2}, but we need to choose different admissible pairs. For $\varepsilon>0$ sufficiently small, we set
\begin{align*}
q_\varepsilon\;=\;\frac{8}{N-4-2\varepsilon}\;,\;\;\;\;r_\varepsilon=\frac{N}{2+\varepsilon},
\end{align*}
and
\begin{align*}
\bar{a}\;=\;\frac{8(\alpha-\theta)}{8-N+2\varepsilon}\;,\;\;\;\;\;\;\;\bar{r}=\frac{2\alpha N(\alpha-\theta)}{\alpha(N-2b-2\varepsilon)-2\theta (4-b)}.
\end{align*}
Note that $r_\varepsilon<\frac{N}{2}$. Moreover, since $N<8$ and $b<\frac{N}{2}$ we get that the denominators of $\bar{a}$ and $\bar{r}$ are positive, if $\theta$ and $\varepsilon$ are sufficiently small. Hence, an easy computation shows that $(\bar{a},\bar{r})$ is $\dot{H}^{s_c}$-biharmonic admissible and $(q_\varepsilon,r_\varepsilon)$ is $B$-admissible. Let $E(t)=\left\|\nabla(|x|^{-b}|u|^\alpha u)\right\|_{L_x^{\frac{2N}{N+2}}(A)}$, where $A$ denotes either $B$ or $B^C$. The H\"older inequality and the Sobolev embedding lead to
\[
\begin{split}\label{L>=8-1}
E(t )&\leq  \||x|^{-b}\|_{L^\gamma(A)}\|u\|^\theta_{L_x^{\theta r_1}}  \|u\|^{\alpha-\theta}_{L_x^{\bar{r}}}   \| \nabla u \|_{L_x^{r_3}}  + \||x|^{-b-1}\|_{L^d(A)}\|u\|^\theta_{L_x^{\theta r_1}}\|u\|^{\alpha-\theta}_{L^{\bar{r}}_x} \|u\|_{L_x^{p_3}}   \\
&\lesssim   \|u\|^\theta_{L_x^{\theta r_1}}  \|u\|^{\alpha-\theta}_{L_x^{\bar{r}}} \| \Delta u \|_{L_x^{r_\varepsilon}} + \|u\|^\theta_{L_x^{\theta r_1}}\|u\|^{\alpha-\theta}_{L^{\bar{r}}_x} \|\Delta u\|_{L_x^{r_\varepsilon}},
\end{split}
\]
where 
\begin{equation*} 
\begin{cases}
\frac{N}{\gamma}=\frac{N}{2}+1-\frac{N}{r_1}-\frac{N(\alpha-\theta)}{r}-\frac{N}{r_3}=\frac{N}{2}+1-\frac{N}{r_1}-\frac{N(\alpha-\theta)}{r}-(\frac{N}{r_\varepsilon}-1)\\ 
\frac{N}{d}=\frac{N}{2}+1-\frac{N}{r_1}-\frac{N(\alpha-\theta)}{r}-\frac{N}{e}=\frac{N}{2}+1-\frac{N}{r_1}-\frac{N(\alpha-\theta)}{r}-(\frac{N}{r_\varepsilon}-2).
\end{cases}
\end{equation*}
Using the definition of the numbers $\bar{r}$ and $r_\varepsilon$ one has
$$
\frac{N}{\gamma}-b=\frac{\theta(4-b)}{\alpha}-\frac{N}{r_1}, \qquad \frac{N}{d}-b-1=\frac{\theta(4-b)}{\alpha}-\frac{N}{r_1},
$$
which are the same relations as in \eqref{L2E}. Since $\frac{1}{2}=\frac{\alpha-\theta}{\bar{a}}+\frac{1}{q_\varepsilon}$, the rest of the proof runs as in Lemma \ref{lemmaglobal2}.
\end{proof}	 

\ It is worth mentioning that assumption $\alpha<\frac{N-2b}{N-4}$ in last lemma appears in view of the condition $\bar{r}<\frac{2N}{N-4}$, which is necessary to $(\bar{a}, \bar{r})$ be $\dot{H}^{s_c}$-biharmonic admissible. However, if we insist with $\alpha$ in the intercritical range, that is, $\frac{8-2b}{N}<\alpha<\frac{8-2b}{N-4}$ one should take $b$ smaller than in the previous lemma. At least in dimension $N=6$ or $N=7$, this is the content of our next lemma.

\begin{lemma}\label{lemmaglobal4} 
Assume $N=6,7$, $0<b<N-4$, and
$
\frac{8-2b}{N}<\alpha<4^*.
$
Then, there exist $\sigma\in (0,1)$ such that
\[
\begin{split}
\left\|\nabla\left(|x|^{-b}|u|^\alpha u \right)\right\|_{L_t^2L_x^{\frac{2N}{N+2}}}&\leq  c\| u\|^{\theta}_{L^\infty_tH^2_x}\|u\|^{\alpha-\theta}_{B(\dot{H}^{s_c})}\left(\| \Delta u\|_{B(L^2)}+\| u\|_{B(L^2)}\right)\nonumber \\
& \quad +c\| u\|^{1-\sigma}_{L_t^\infty H^2_x}\|u\|^{\theta^*}_{B(\dot{H}^{s_c})} \|\Delta u\|^{\alpha-\theta^*+\sigma}_{B(L^2)},
\end{split}
\]
where $\theta^*=\alpha F$, with $F=\frac{4-2b-2\varepsilon+\sigma(N-4)}{8-2b}$, and $\varepsilon,\theta>0$ are sufficiently small numbers.
\end{lemma}
\begin{proof} 
Let us start by defining the following numbers  		
\begin{equation}\label{paradmissivel1}
k=\frac{8\alpha(\alpha+1-\theta)}{8-2b-\alpha(N-4)},\hspace{1.5cm}\;p=\frac{2N\alpha(\alpha+1-\theta)}{(8-2b)(\alpha-\theta)+\alpha(N-4)},
\end{equation}
and
\begin{equation}\label{paradmissivel2}
l=\frac{8\alpha(\alpha+1-\theta)}{\alpha(N\alpha-4+2b)-\theta(N\alpha-8+2b)},
\end{equation}	 
where $\theta\in (0,\alpha)$. It follows easily that $(l,p)$ is $B$-admissible and $(k,p)$ is $\dot{H}^{s_c}$-biharmonic admissible.

Next, take $\sigma\in(0,1)$ sufficiently close to 1 such that $\frac{b}{N-4}<\sigma<1$. For  $D=\alpha-\theta^*+\sigma$    and $\varepsilon>0$ sufficiently small, we set
\begin{equation}\label{PAL41} 
m=\frac{8D}{D(N-4)-2\varepsilon},\;\hspace{1.5cm}\;n=\frac{ND}{2D+\varepsilon},
\end{equation}
and 
\begin{equation}\label{PAL42}
a^*=\frac{8\theta^*}{4+2\varepsilon-D(N-4)}, \;\;\;\;\quad r^*=\frac{2N\alpha\theta^*}{(8-2b)\theta^*-(4+2\varepsilon-D(N-4))\alpha}.
\end{equation}
By assuming that $\varepsilon$ is sufficiently small such that $\sigma(N-4)-b>2\varepsilon$ and $\sigma(N-4)<4+2\varepsilon$ we promptly deduce that $F\in(\frac{1}{2},1)$. In particular, we have $D=\alpha(1-F)+\sigma>0$. So, after some calculations, we deduce  that $(m,n)$ is $B$-admissible and   $(a^*,r^*)$ is $\dot{H}^{s_c}$-biharmonic admissible.	
	
Next we will get the estimate in the lemma itself. Indeed, observe that 
\begin{equation}\label{m3}
\left\|\nabla\left(|x|^{-b}|u|^\alpha u \right)\right\|_{L_t^2L_x^{\frac{2N}{N+2}}}\leq \left\|\nabla\left(|x|^{-b}|u|^\alpha u \right)\right\|_{L_t^2L_x^{\frac{2N}{N+2}}(B)}+\left\|\nabla \left(|x|^{-b}|u|^\alpha u \right)\right\|_{L_t^2L_x^{\frac{2N}{N+2}}(B^C)}.
\end{equation}
As before, let $A$ denote either $B$ or $B^C$. We then have
\begin{equation}\label{LG4Hs1} 
 \left\|\nabla\left(|x|^{-b}|u|^\alpha u \right) \right \|_{L_x^{\frac{2N}{N+2}}(A)}\leq M_1(t,A)+M_2(t,A), 
\end{equation}  
 where 
 $$
 M_1(t,A)=\left \||x|^{-b} \right \|_{L^\gamma(A)} \left \|\nabla(|u|^\alpha u)\right\|_{L^{\beta}_x}, \quad \mbox{and} \quad M_2(t,A)=\left\|\nabla(|x|^{-b})\right\|_{L^d(A)} \left\||u|^\alpha u\right\|_{L^{e}_x}
 $$    
with
\begin{equation}\label{LG4Hs2}
\frac{N+2}{2N}=\frac{1}{\gamma}+\frac{1}{\beta}=\frac{1}{d}+\frac{1}{e}. 
 \end{equation}   
	      
We start by estimating  $M_1(t,A)$. The H\"older inequality and the Sobolev embedding yield 
\begin{equation}\label{LG4Hs3}
 \begin{split}
 M_1(t,A) &\lesssim\||x|^{-b}\|_{L^\gamma(A)}  \|u\|^{\theta}_{L^{\theta r_1}_x}   \|u\|^{\alpha-\theta}_{L_x^{(\alpha-\theta)r_2}}  \|\nabla u\|_{ L^{r_3}_x}  \\
 &\lesssim    \||x|^{-b}\|_{L^\gamma(A)}   \|u\|^{\theta}_{L^{\theta r_1}_x}  \|u\|^{\alpha-\theta}_{L_x^p}  \|\Delta u\|_{ L^{p}_x},  
 \end{split}
\end{equation}
where 
\begin{equation}\label{LG4Hs31}
 \frac{1}{\beta}=\frac{1}{r_1}+\frac{\alpha-\theta}{p}+\frac{1}{r_3},\;\;\;\;\textnormal{and}\;\;\;\; 1=\frac{N}{p}-\frac{N}{r_3},\,\;\;p<N.
 \end{equation}

We notice that, for $\theta$ sufficiently small,  $p<N$ is equivalent to $\alpha<\frac{N+2-2b}{2}$, which is true thanks to our assumptions $\alpha<4^*$.  From \eqref{LG4Hs2} and \eqref{LG4Hs31} we obtain
 \begin{equation*}
 \frac{N}{\gamma}=\frac{N}{2}+2-\frac{N}{r_1}-\frac{N(\alpha+1-\theta)}{p},
\end{equation*}
 which implies, by \eqref{paradmissivel1},
\begin{equation*}\label{LG4Hs32}
\frac{N}{\gamma}-b=\frac{\theta(4-b)}{\alpha}-\frac{N}{r_1}.
\end{equation*}
Now let us check that $\||x|^{-b}\|_{L^\gamma(A)}$ is finite by choosing $r_1$ in an appropriate manner. In fact, if $A=B$,  we choose $\theta r_1=\frac{2N}{N-4}$, so that
$\frac{N}{\gamma}-b=\theta(2-s_c)>0.$
On the other hand, if $A=B^C$, we choose $\theta r_1=2$, so that
$\frac{N}{\gamma}-b=-\theta s_c<0.$
Therefore,  inequality \eqref{LG4Hs3} and the Sobolev embedding yield
\begin{equation}\label{M1tA}
 M_1(t,A) \lesssim \|u\|^{\theta}_{H^2_x}\|u\|^{\alpha-\theta}_{L_x^p}  \|\Delta u\|_{ L^p_x}.  
\end{equation}
	     
\  We now estimate $M_2(t,A)$. Assume first that $A=B^C$.  By applying the H\"older inequality one has	
\begin{equation*}
 M_2(t,B^C)
\leq   \||x|^{-b-1} \|_{L^d(B^C)} \|u\|^{\theta}_{L^{\theta r_1}_x}   \|u\|^{\alpha-\theta}_{L_x^p} \| u\|_{L_x^p},
\end{equation*}
where $\frac{1}{e}=\frac{1}{r_1}+\frac{\alpha-\theta}{p}+\frac{1}{p}$. The last relation and \eqref{LG4Hs2} imply
 \begin{equation*}
 \frac{N}{d}=\frac{N+2}{2}-\frac{N}{r_1}-\frac{N(\alpha+1-\theta)}{p}.
\end{equation*}
In view of \eqref{paradmissivel1} we deduce  $\frac{N}{d}=-1+b+\frac{\theta(4-b)}{\alpha}-\frac{N}{r_1}$.
Setting $\theta r_1=\frac{2N}{N-4}$ we have 
$$\frac{N}{d}-b-1=-2(1-\theta)-\theta s_c<0,$$
implying that $|x|^{-b-1}\in L^d(B^C)$. So, by the Sobolev embedding,
\begin{equation}\label{M2tBC}
 M_2(t,B^C) \lesssim \|u\|^{\theta}_{H^2_x}   \|u\|^{\alpha-\theta}_{L_x^p} \| u\|_{L_x^p}.
\end{equation}
Assume now $A=B$. From the H\"older inequality and the Sobolev embedding (note that $n<\frac{N}{2}$)	
\begin{equation}\label{m2}
\begin{split}
 M_2(t,B)& \lesssim \||x|^{-b-1} \|_{L^d(B)} \|u\|^{\theta^*}_{L^{\theta^*r_1}_x}   \|u\|^{\alpha-\theta^*+\sigma}_{L_x^{(\alpha-\theta^*+\sigma)r_2}}\| u\|^{1-\sigma}_{L_x^{(1-\sigma)r_3}}  \\ 
&\lesssim   \||x|^{-b-1}\|_{L^d(B)}  \|u\|^{\theta^*}_{L^{r^*}_x} \|\Delta u\|^{\alpha-\theta^*+\sigma}_{L_x^n} \| u\|^{1-\sigma}_{L_x^{(1-\sigma)r_3}},
\end{split}
\end{equation}
where
\[
\frac{1}{e}=\frac{\theta^*}{r*}+\frac{1}{r_2}+\frac{1}{r_3} \quad \mbox{and}\quad
 2=\frac{N}{n}-\frac{N}{(\alpha-\theta^*+\sigma) r_2}.
\]
It follows from \eqref{LG4Hs2} that (recalling $D=\alpha-\theta^*+\sigma$)
\begin{equation}\label{LG4Hs5}
\frac{N}{d}= \frac{N+2}{2}+2D- \frac{N\theta^*}{r^*}-\frac{ND}{n}-\frac{N}{r_3},
\end{equation}
which implies by \eqref{PAL41}, \eqref{PAL42} and choosing $(1-\sigma)r_3=\frac{2N}{N-4}$
\begin{eqnarray*}
\frac{N}{d}-b-1=4-b-\frac{\theta^*(4-b)}{\alpha}-\frac{(\alpha-\theta^*)(N-4)}{2}
=(2-s_c)(\alpha-\theta^*).
\end{eqnarray*}
Since $s_c<2$ and $\alpha-\theta^*=\alpha(1-F)>0$, we get $\frac{N}{d}-b-1>0$ and so $|x|^{-b-1}\in L^d(B)$. Hence,  from \eqref{m2} and Sobolev's embedding,
\begin{equation}\label{M2tB}
 M_2(t,B)\lesssim \|u\|^{1-\sigma}_{H^2_x}\|u\|^{\theta^*}_{L^{r*}_x}   \|\Delta u\|^{\alpha-\theta^*+\sigma}_{L_x^{n}}.
\end{equation}

\ Therefore, gathering together the above estimates (see \eqref{M1tA}, \eqref{M2tBC}, and \eqref{M2tB}), we obtain
$$
\left\| \nabla\left(|x|^{-b}|u|^\alpha u \right) \right \|_{L_x^{\frac{2N}{N+2}}(B^C)}\lesssim \|u\|^{\theta}_{H^2_x}   \|u\|^{\alpha-\theta}_{L_x^p} \|\Delta u\|_{L_x^p}+ \|u\|^{\theta}_{H^2_x}   \|u\|^{\alpha-\theta}_{L_x^p} \| u\|_{L_x^p}
$$
and
$$
\left\| \nabla\left(|x|^{-b}|u|^\alpha u \right) \right \|_{L_x^{\frac{2N}{N+2}}(B)}\lesssim \|u\|^{\theta}_{H^2_x}   \|u\|^{\alpha-\theta}_{L_x^p} \| \Delta u\|_{L_x^p}+ \|u\|^{1-\sigma}_{H^2_x}\|u\|^{\theta^*}_{L^{r*}_x}   \|\Delta  u\|^{\alpha-\theta^*+\sigma}_{L_x^{n}}.
$$

\ Finally, since 
$
\frac{1}{2}=\frac{\alpha-\theta}{k}+\frac{1}{l}
$
and 
$
\frac{1}{2}=\frac{\theta}{a^*}+\frac{\alpha-\theta+\sigma}{m},
$
we can use H\"older's inequality in the time variable in the last two inequalities to conclude
\begin{equation*}
 \left\|  \nabla\left(|x|^{-b}|u|^\alpha u \right) \right \|_{L^{2}_tL_x^{\frac{2N}{N+2}}(B^C)}  \lesssim \|u\|^{\theta}_{L^\infty_tH^2_x}\|u\|^{\alpha-\theta}_{L^k_tL_x^p} \left( \|\Delta u\|_{L^l_t L^p_x}+ \| u\|_{L^l_t L^p_x}\right) \\
\end{equation*}
and
\[
\begin{split}
 \left\|  \nabla \left(|x|^{-b}|u|^\alpha u \right) \right \|_{L^{2}_tL_x^{\frac{2N}{N+2}}(B)}  &\lesssim \|u\|^{\theta}_{L^\infty_tH^2_x}\|u\|^{\alpha-\theta}_{L^k_tL_x^p}  \|\Delta u\|_{L^l_t L^p_x}\nonumber  \\
 &\quad + \|u\|^{1-\sigma}_{L^\infty_tH^2_x}\|u\|^{\theta^*}_{L^{a^*}_tL^{r*}_x}   \|\Delta u\|^{\alpha-\theta^*+\sigma}_{L^m_tL_x^{n}}.
\end{split}
\]
In view of \eqref{m3} and recalling that $(m,n)$ and $(l,p)$ are $B$-admissible and $(k,p)$ and $(a^*,r^*)$ are $\dot{H}^{s_c}$-admissible, the proof is completed.
\end{proof}	

\begin{remark}\label{remn=5}
In dimension $N=5$, our proof of Lemma \ref{lemmaglobal4}  does not produce a better result than in Lemma \ref{lemmaparcial}. Indeed, in order to use the Sobolev embedding we have strongly used that $p<N$ is equivalent to $\alpha<\frac{N+2-2b}{2}$. Hence, if $N=5$ we would have $\alpha<\frac{7-2b}{2}$. On the other hand, in dimension $N=5$, Lemma \ref{lemmaparcial} (ii) holds for $\alpha<5-2b$.
\end{remark}

\ Finally we treat the cases $N=3,4$.

\begin{lemma}\label{lemmaglobal3} 
Let $N=3,4$ and $0<b<\frac{N}{2}$. If $ \frac{8-2b}{N}<\alpha<\infty$, then if $F(x,u)=|x|^{-b}|u|^\alpha u$,
\begin{itemize}
\item [(i)] 
 $\|\nabla F\|_{L_t^2L_x^{\frac{2N}{N+2}}}\leq  c\| u\|^{\theta^*}_{L^\infty_tH^2_x}\|u\|^{\alpha-\theta^*}_{B(\dot{H}^{s_c})}\| \Delta u\|_{B(L^2)}^{1/2}\| u\|_{B(L^2)}^{1/2}+c\| u\|^{\theta+1}_{L_t^\infty H^2_x}\|u\|^{\alpha-\theta}_{B(\dot{H}^{s_c})},\;  \mbox{if}\;N=4
$
\item[(ii)]
$\|\nabla F\|_{L_t^2L_x^{\frac{2N}{N+2}}}\leq  c\| u\|^{\theta^*}_{L^\infty_tH^2_x}\|u\|^{\alpha-\theta^*}_{B(\dot{H}^{s_c})}\| \Delta u\|_{B(L^2)}^{1/2}\| u\|_{B(L^2)}^{1/2}+c\| u\|^{\bar{\theta}+1}_{L_t^\infty H^2_x}\|u\|^{\alpha-\bar{\theta}}_{B(\dot{H}^{s_c})}
$

\ 

$
 \qquad \qquad + c\| u\|^{\theta}_{L^\infty_tH^2_x}\|u\|^{\alpha-\theta}_{B(\dot{H}^{s_c})}\| \Delta u\|_{B(L^2)}^{1/2}\| u\|_{B(L^2)}^{1/2}+c\| u\|^{\theta+1}_{L_t^\infty H^2_x}\|u\|^{\alpha-\theta}_{B(\dot{H}^{s_c})} , \quad \mbox{if}\;N=3$,
\end{itemize}
where $c>0$, $\theta\in (0,\alpha)$ is sufficiently small, $\theta^*=\frac{\alpha}{4-b}+$ and $\bar{\theta}=\frac{\alpha}{8-2b}+$ .
\end{lemma}
\begin{proof} 
Define the numbers
\begin{equation*}\label{PAN=4}
a=\frac{8\alpha(\alpha+1-\theta^*)}{8-2b-\alpha(N-4)}, \qquad r=\frac{2N\alpha(\alpha+1-\theta^*)}{(8-2b)(\alpha-\theta^*)+\alpha(N-4)}, 
\end{equation*}
and
\begin{equation*}
q=\frac{8\alpha(\alpha+1-\theta^*)}{\alpha(N(\alpha-\theta^*)+4)-(\alpha-\theta^*)(8-2b)}.
\end{equation*}
We may check that $(a,r)$ is $\dot{H}^{s_c}$-admissible, $(q,r)$ is $B$-admissible and
\begin{equation}\label{relal}
\frac{\alpha-\theta^*}{a}+\frac{1}{q}=\frac{1}{2}.
\end{equation}
\item [(i)] {\bf Case $N=4$.} In this case, we will also use the  $\dot{H}^{s_c}$-admissible pair given by
\begin{equation*}\label{PAN=41}
\widetilde{a}=2(\alpha-\theta), \qquad \widetilde{r}=\frac{4\alpha(\alpha-\theta)}{\alpha(2-b)-\theta(4-b)}.
\end{equation*}
As in Lemma \ref{lemmaglobal4}, we note that
\begin{equation}\label{m31}
\left\|\nabla\left(|x|^{-b}|u|^\alpha u \right)\right\|_{L_t^2L_x^{\frac{2N}{N+2}}}\leq \left\|\nabla\left(|x|^{-b}|u|^\alpha u \right)\right\|_{L_t^2L_x^{\frac{2N}{N+2}}(B)}+\left\|\nabla \left(|x|^{-b}|u|^\alpha u \right)\right\|_{L_t^2L_x^{\frac{2N}{N+2}}(B^C)}.
\end{equation}
Now we estimate both terms on the right-hand side of \eqref{m31}. Let $A$ denote either $B$ or $B^C$. From H\"older's inequality we obtain
\[
\begin{split}
\left\|\nabla\left(|x|^{-b}|u|^\alpha u \right)\right\|_{L_x^{\frac{2N}{N+2}}(A)}
&\lesssim \||x|^{-b}\|_{L^\gamma(A)}\|\nabla(|u|^\alpha u)\|_{L^\beta_x} + \||x|^{-b-1}\|_{L^d(A)}\||u|^\alpha u\|_{L^e_x}\\
&\lesssim \||x|^{-b}\|_{L^\gamma(A)}\|u\|_{L_{x}^{r_1\theta^*}}^{\theta^*}\|u\|_{L_{x}^{(\alpha-\theta^*)r_2}}^{\alpha-\theta^*}\|\nabla u \|_{L^{r_3}_x}\\
&\quad + \||x|^{-b-1}\|_{L^d(A)}\|u\|_{L_{x}^{(\theta+1)p_1}}^{\theta+1} \|u\|_{L_{x}^{(\alpha-\theta)p_2}}^{\alpha-\theta},
\end{split}
\]
where
\begin{equation}\label{m4}
\frac{N+2}{2N}=\frac{1}{\gamma}+\frac{1}{\beta} =\frac{1}{\gamma}+\frac{1}{r_1}+\frac{1}{r_2}+\frac{1}{r_3} \qquad \mbox{and}\qquad  \frac{N+2}{2N}= \frac{1}{d}+\frac{1}{e}=\frac{1}{d}+\frac{1}{p_1}+\frac{1}{p_2}.
\end{equation}
By choosing $(\alpha-\theta^*)r_2=r$, $r_3=r$, $(\alpha-\theta)p_2=\widetilde{r}$, and using H\"older's inequality in time (recall \eqref{relal}), we infer
\begin{equation*}
\begin{split}
\left\|\nabla\left(|x|^{-b}|u|^\alpha u \right)\right\|_{L_t^2L_x^{\frac{2N}{N+2}}(A)}& 
\lesssim \||x|^{-b}\|_{L^\gamma(A)}\|u\|_{L^\infty_tL_{x}^{r_1\theta^*}}^{\theta^*}\|u\|_{L^a_tL_{x}^{r}}^{\alpha-\theta^*}\|\nabla u \|_{L^q_tL^{r}_x}\\
&\quad + \||x|^{-b-1}\|_{L^d(A)}\|u\|_{L^\infty_tL_{x}^{(\theta+1)p_1}}^{\theta+1} \|u\|_{L_t^{\widetilde{a}}L_{x}^{\widetilde{r}}}^{\alpha-\theta} \\
&\lesssim \||x|^{-b}\|_{L^\gamma(A)}\|u\|_{L^\infty_tL_{x}^{r_1\theta^*}}^{\theta^*}\|u\|_{L^a_tL_{x}^{r}}^{\alpha-\theta^*}\| u \|_{L^q_tL^{r}_x}^{1/2}\|\Delta u \|_{L^q_tL^{r}_x}^{1/2}\\
&\quad + \||x|^{-b-1}\|_{L^d(A)}\|u\|_{L^\infty_tL_{x}^{(\theta+1)p_1}}^{\theta+1} \|u\|_{L_t^{\widetilde{a}}L_{x}^{\widetilde{r}}}^{\alpha-\theta}. 
\end{split}
\end{equation*}
In order to finish the proof of part (i) it is sufficient to check that $\||x|^{-b}\|_{L^\gamma(A)}$ and $\||x|^{-b-1}\|_{L^d(A)}$ are finite and $L_x^{r_1\theta^*}(\mathbb{R}^4)$ and  $L_x^{(\theta+1)p_1}(\mathbb{R}^4)$ are embedded in $H^2(\mathbb{R}^4)$. For this, we will choose the parameters $r_1$ and $p_1$ appropriately. 
From \eqref{m4} (recalling that $N=4$),
$$
\frac{4}{d}=3-\frac{4}{p_1}-\frac{4(\alpha-\theta)}{\widetilde{r}}=1+b+\frac{\theta(4-b)}{\alpha}-\frac{4}{p_1}.
$$
Hence $\frac{4}{d}-b-1>0$ if and only if $p_1>\frac{4\alpha}{\theta(4-b)}$. By observing that $\frac{4\alpha(\theta+1)}{\theta(4-b)}>2$ (in view of our assumption $\alpha>\frac{8-2b}{4}$) we then see that if $A=B$ it is sufficient to choose $p_1>\frac{4\alpha}{\theta(4-b)}$ and if $A=B^C$ it is sufficient to choose $p_1$ such that $(\theta+1)p_1\in \left(2, \frac{4\alpha(\theta+1)}{\theta(4-b)}\right)$. In both cases we have $\||x|^{-b-1}\|_{L^d(A)}$ finite and $(\theta+1)p_1>2$, from which we obtain $H^2(\mathbb{R}^4)\hookrightarrow L^{(\theta+1)p_1}(\mathbb{R}^4)$. 

Also from \eqref{m4},
$$
\frac{4}{\gamma}=3-\frac{4}{r_1}-\frac{4(\alpha+1-\theta^*)}{r}=b-1+\frac{\theta^*(4-b)}{\alpha}-\frac{4}{r_1}.
$$
Since $\theta^*$ is slightly bigger than $\frac{\alpha}{4-b}$, let us write $\theta^*=\frac{\alpha}{4-b-3\alpha\delta}$, where $\delta>0$ is sufficiently small. If $A=B$ we choose $r_1$ such that $r_1\theta^*=\frac{\delta}{2}$ to deduce that $\frac{4}{\gamma}-b=\delta\theta^*>0$. On the other hand, if $A=B^C$ we choose $r_1=\frac{2}{\theta_1}$ to obtain that $\frac{4}{\gamma}-b=3\delta\theta^*-2\theta^*$, which is negative because $\delta$ is sufficiently small.  Thus, in both cases we have $\||x|^{-b}\|_{L^\gamma(A)}$ finite and $r_1\theta^*\geq2$, from which we also obtain $H^2(\mathbb{R}^4)\hookrightarrow L^{r_1\theta^*}(\mathbb{R}^4)$. This completes the proof of (i).

\ 

\item [(ii)] {\bf Case $N=3$.} First, we consider the estimate on $B$. As  before, (using $H^2\hookrightarrow L^\infty$ and \eqref{relal})
\[
\begin{split}
\left\|\nabla\left(|x|^{-b}|u|^\alpha u \right)\right\|_{L^2_tL_x^{\frac{2N}{N+2}}(B)}
&\leq \left \| \||x|^{-b}\|_{L^\gamma(A)}\|u\|_{L_{x}^{\infty}}^{\theta^*}\|u\|_{L_{x}^{r}}^{\alpha-\theta^*}\|\nabla u\|_{L^{r}_x}\right\|_{L^2_t}\\
&\quad + \left\|\||x|^{-b-1}\|_{L^d(A)}\|u\|_{L_{x}^{\infty}}^{\bar{\theta}+1} \|u\|_{L_{x}^{\widetilde{r}}}^{\alpha-\bar{\theta}}\right\|_{L^2_t}\\
&\lesssim  \|u\|^{\theta^*}_{L^\infty_tH_{x}^{2}}\|u\|_{L_t^{a}L_{x}^{r}}^{\alpha-\theta^*}\|\nabla u\|_{L_t^{q}L^{r}_x} + \|u\|_{L^\infty_tH_{x}^{2}}^{\bar{\theta}+1} \|u\|_{L_{t}^{\bar{a}}L_{x}^{\bar{r}}}^{\alpha-\bar{\theta}},
\end{split} 
\]
where $(\bar{a},\bar{r})$ is the $\dot{H}^{s_c}$-admissible pair given by  $\left(2(\alpha-\bar{\theta}), \frac{3\alpha(\alpha-\bar{\theta})}{\alpha(2-b)-\bar{\theta}(4-b)}\right)$ and
\begin{equation*}
\frac{N+2}{2N}=\frac{1}{\gamma}+\frac{\alpha-\theta^*}{r}+\frac{1}{r}, \qquad   \frac{N+2}{2N}=\frac{1}{d}+\frac{\alpha-\bar{\theta}}{\bar{r}},
\end{equation*}
which implies that  $\frac{3}{\gamma}-b=-1+\frac{\theta^*(4-b)}{\alpha}$ and $\frac{3}{d}-b-1=-\frac{1}{2}+\frac{\bar{\theta}(4-b)}{\alpha}$. It follows from $\theta^*=\frac{\alpha}{4-b}+$ and $\bar{\theta}=\frac{\alpha}{8-2b}+$ that $\frac{3}{\gamma}-b>0$ and $\frac{3}{b}-b-1>0$, i.e., the norms $\||x|^{-b}\|_{L^\gamma(B)}$ and $\||x|^{-b-1}\|_{L^d(B)}$ are finite. Thus,
\[
\begin{split}
\left\|\nabla\left(|x|^{-b}|u|^\alpha u \right)\right\|_{L^2_tL_x^{\frac{2N}{N+2}}(B)}\lesssim \|u\|^{\theta^*}_{L^\infty_tH_{x}^{2}}\|u\|_{B(\dot{H}^{s_c})}^{\alpha-\theta^*}\|\nabla u\|_{L_t^{q}L^{r}_x}+ \|u\|_{L^\infty_tH_{x}^{2}}^{\bar{\theta}+1} \|u\|_{B(\dot{H}^{s_c})}^{\alpha-\bar{\theta}}.
\end{split}
\]

\ We now estimate on $B^C$. Arguing in the same way as before and using $\theta$ instead of $\theta^*$ and $\bar{\theta}$, we see if $\theta$ is small, then $\||x|^{-b}\|_{L^\gamma(B)}, \||x|^{-b-1}\|_{L^d(B)}<\infty$. Thus,
$$
\left\|\nabla\left(|x|^{-b}|u|^\alpha u \right)\right\|_{L^2_tL_x^{\frac{2N}{N+2}}(B^C)}\lesssim \|u\|^{\theta}_{L^\infty_tH_{x}^{2}}\|u\|_{B(\dot{H}^{s_c})}^{\alpha-\theta}\|\nabla u\|_{L_t^{q}L^{r}_x}+ \|u\|_{L^\infty_tH_{x}^{2}}^{\theta+1} \|u\|_{B(\dot{H}^{s_c})}^{\alpha-\theta}.
$$
We complete the proof of the lemma using the last two inequalities and interpolation.
\end{proof}

Now, with all the previous lemmas in hand we are in a position to prove Theorem \ref{GWPH2}. 

\begin{proof}[\bf{Proof of Theorem \ref{GWPH2}}] 
As before, we use the contraction mapping argument to the map $G$. 
Let $S$ be the set of all functions $u:\mathbb{R}^N\times \mathbb{R} \to\mathbb{R}$ such that
$$
\|u\|_{B(\dot{H}^{s_c})}\leq 2\|e^{it\Delta^2}u_0\|_{B(\dot{H}^{s_c})}\;\quad \textnormal{and}\; \quad \|u\|_{B(L^2)}+\|\Delta u\|_{B(L^2)}\leq 2c\|u_0\|_{H^2}.
$$ 
We shall show that $G=G_{u_0}$ defined in \eqref{OPERATOR} is a contraction on $S$ equipped with the metric 
$$
d(u,v)=\|u-v\|_{B(L^2)}+\|u-v\|_{B(\dot{H}^{s_c})}.
$$

Assume first that condition (i) of the theorem holds.	
By using Lemma \ref{Lemma-Str} and Proposition \ref{estimativanaolinear} (see \eqref{EstimativaImportante},) we get
\begin{equation*}\label{GHs11}
\|G(u)\|_{B(\dot{H}^{s_c})}\leq \|e^{it\Delta^2}u_0\|_{B(\dot{H}^{s_c})}+ c\| \chi_BF \|_{B'(\dot{H}^{-s_c})}+c\| \chi_{B^C}F \|_{B'(\dot{H}^{-s_c})}
\end{equation*}
\begin{equation*}\label{GHs21}
\|G(u)\|_{B(L^2)}\leq c\|u_0\|_{L^2}+ c\|\chi_{B} F \|_{B'(L^2)}+ c\|\chi_{B^C} F \|_{B'(L^2)}
\end{equation*}
and	
\begin{equation*}\label{GHs31}
\|\Delta G(u)\|_{B(L^2)}\leq c \|\Delta u_0\|_{L^2}+ c\|\nabla F\|_{L^2_tL_x^{\frac{2N}{N+2}}},
\end{equation*}
where  $F=F(x,u)=|x|^{-b}|u|^\alpha u$. 
An application of Lemmas \ref{lemmaglobal1} and \ref{lemmaglobal2} then yield, for any $u\in S$,
\begin{equation}\label{TGHS1}
\begin{split}
\|G(u)\|_{B(\dot{H}^{s_c})}
&\leq \|e^{it\Delta^2}u_0\|_{B(\dot{H}^{s_c})} +c\| u \|^\theta_{L^\infty_tH^2_x}\| u \|^{\alpha-\theta}_{B(\dot{H}^{s_c})}\| u \|_{B(\dot{H}^{s_c})}  \\
&\leq  \|e^{it\Delta^2}u_0\|_{S(\dot{H}^{s_c})}+2^{\alpha+1}c^{\theta+1}\|u_0\|^\theta_{H^2}\| e^{it\Delta^2}u_0 \|^{\alpha-\theta+1}_{B(\dot{H}^{s_c})} \\ 
&\leq  \|e^{it\Delta^2}u_0\|_{S(\dot{H}^{s_c})}+2^{\alpha+1}c^{\theta+1}\eta^\theta\| e^{it\Delta^2}u_0 \|^{\alpha-\theta+1}_{B(\dot{H}^{s_c})}
\end{split}
\end{equation}
and
\begin{equation}\label{TGHS11}
\begin{split}
\|G(u)\|_{B(L^2)}+\|\Delta G(u)\|_{B(L^2)}&\leq  c\|u_0\|_{H^2}+c\| u \|^\theta_{L^\infty_tH^2_x}\| u \|^{\alpha-\theta}_{B(\dot{H}^{s_c})}(\|\Delta u\|_{B(L^2)}+\|u\|_{B(L^2)})\\
& \leq  c\|u_0\|_{H^2}+c2^{\alpha+1}c^{\theta+1}\|u_0\|_{H^2_x}^{\theta+1} \| e^{it\Delta^2}u_0 \|^{\alpha-\theta}_{B(\dot{H}^{s_c})}\\
& \leq  c\|u_0\|_{H^2}+c2^{\alpha+1}c^{\theta+1}\eta^\theta \| e^{it\Delta^2}u_0 \|^{\alpha-\theta}_{B(\dot{H}^{s_c})}\|u_0\|_{H^2_x}
\end{split}
\end{equation}
where in the second inequality we have used the fact that $(\infty,2)$ is $B$-admissible to see that $\|u\|_{L^\infty_tH^2_x}\leq \|\Delta u\|_{B(L^2)}+\|u\|_{B(L^2)}$.

\ Now if $\| e^{it\Delta^2}u_0 \|_{B(\dot{H}^{s_c})}<\delta$ with 
\begin{equation}\label{WD1}
\delta\leq \min\left\{\sqrt[\alpha-\theta]{\frac{1}{2c^{\theta+1}2^{\alpha+1}\eta^\theta}}     , \sqrt[\alpha-\theta]{ \frac{1}{4c^{\theta+1}2^{\alpha+1}\eta^\theta}}\right\},
\end{equation}
 it follows from \eqref{TGHS1} and \eqref{TGHS11} that 
$$\|G(u)\|_{B(\dot{H}^{s_c})}\leq 2\| e^{it\Delta^2}u_0 \|_{B(\dot{H}^{s_c})}\quad \mbox{and}\quad \|G(u)\|_{B(L^2)}+\|\Delta G(u)\|_{B(L^2)}\leq 2c\|u_0\|_{H^2},$$
which means to say   $G(u)\in S$.

\ To show that $G$ is a contraction on $S$, we repeat the above computations taking into account \eqref{FEI}. Indeed,
\begin{equation*}
\begin{split}\label{C1GH1}
\|G(u)-G(v)\|_{B(\dot{H}^{s_c})}&\leq c\|\chi_B(F(x,u)-F(x,v))\|_{B'(\dot{H}^{-s_c})}+ c\|\chi_{B^C}(F(x,u)-F(x,v))\|_{B'(\dot{H}^{-s_c})}\\
&\leq c\left\|\chi_B|x|^{-b}\big||u|+|v|\big|^{\alpha}|u-v|\right\|_{B'(\dot{H}^{-s_c})}+ c\left\|\chi_{B^C}|x|^{-b}\big||u|+|v|\big|^{\alpha}|u-v|\right\|_{B'(\dot{H}^{-s_c})}\\
&\leq c\left( \| u \|^\theta_{L^\infty_tH^2_x}+\| v \|^\theta_{L^\infty_tH^2_x}\right)\left(\| u \|^{\alpha-\theta}_{B(\dot{H}^{s_c})}+\| v \|^{\alpha-\theta}_{B(\dot{H}^{s_c})}\right)\| u -v\|_{B(\dot{H}^{s_c})}.\\
\end{split}
\end{equation*}
Thus, if $u,v\in S$ then
\begin{equation*}
\begin{split}
\|G(u)-G(v)\|_{B(\dot{H}^{s_c})} &\leq  2c(2c)^\theta \| u_0 \|^\theta_{H^2}2^{\alpha-\theta}\|e^{it\Delta^2}u_0 \|^{\alpha-\theta}_{B(\dot{H}^{s_c})}\| u -v\|_{B(\dot{H}^{s_c})}\\
& =   2^{\alpha+1}c^{\theta+1} \| u_0 \|^\theta_{H^2}\|e^{it\Delta^2}u_0 \|^{\alpha-\theta}_{B(\dot{H}^{s_c})}\| u -v\|_{B(\dot{H}^{s_c})}.
\end{split}
\end{equation*}
By similar arguments we also obtain
\begin{equation*}\label{C2GH1}
\|G(u)-G(v)\|_{B(L^2)}\leq 2^{\alpha+1}c^{\theta+1} \| u_0 \|^\theta_{H^2}\|e^{it\Delta^2}u_0 \|^{\alpha-\theta}_{B(\dot{H}^{s_c})} \|u-v\|_{B(L^2)}.
\end{equation*}
From the two last inequalities and \eqref{WD1} it follows that
$$
d(G(u),G(v)) \leq 2^{\alpha+1}c^{\theta+1} \| u_0 \|^\theta_{H^2}\|e^{it\Delta^2}u_0 \|^{\alpha-\theta}_{B(\dot{H}^{s_c})}d(u,v)\leq \frac{1}{2}d(u,v),
$$
which means that  $G$ is also a contraction. Therefore, by the contraction mapping principle, $G$ has a unique fixed point $u\in S$, which is a global solution of \eqref{IBNLS}. Thus the proof of the theorem is completed in this case.

By using Lemmas \ref{lemmaparcial}, \ref{lemmaglobal4} and \ref{lemmaglobal3} instead of Lemma \ref{lemmaglobal2}, the same proof, with minor modifications, still goes if we are in the assumptions\footnote{In (iii) and (iv), we use the fact that if $u\in S$, then $\|u\|^{\frac{1}{2}}_{B(L^2)}\|\Delta u\|^{\frac{1}{2}}_{B(L^2)}\leq 2c\|u_0\|_{H^2}$. } (ii), (iii) and (iv) of the theorem. So we omit the details.
\end{proof}

\subsection{Scattering}

\ As mentioned in the introduction, Proposition \ref{SCATTERSH1} gives us a criterion to establish scattering.  Before proving the proposition itself, we must point out that our estimates in Lemmas \ref{lemmaglobal1}, \ref{lemmaglobal2}, \ref{lemmaparcial}, \ref{lemmaglobal4}, and \ref{lemmaglobal3} also hold if we replace the norms (in time) on the whole $\mathbb{R}$ by a bounded interval, say, $I$. To see this it is sufficient to note that in all results the only estimates in time we used was the H\"older inequality.

\begin{proof}[\bf{Proof of Proposition \ref{SCATTERSH1}}] First, we claim that  $\|u\|_{B(\dot{H}^{s_c})}<+\infty$ implies
 \begin{equation}\label{SCATTER1}
 \|u\|_{B(L^2)}+\|\Delta u\|_{B(L^2)}<+\infty.
 \end{equation}
We will only show that  $ \|u\|_{B(L^2;[0,\infty))}+\|\Delta u\|_{B(L^2;[0,\infty))}<+\infty$. A similar analysis may be performed to see that $ \|u\|_{B(L^2;(-\infty,0])}+\|\Delta u\|_{B(L^2;(\infty,0])}<+\infty$. Given $\delta>0$ (to be chosen later) we decompose the interval $[0,\infty)$ into $n$ intervals $I_j=[t_j,t_{j+1})$ such that $\|u\|_{B(\dot{H}^{s_c};I_j)}<\delta$, for all $j=1,\ldots,n$. The integral equation  on the time interval $I_j$ is given by
\begin{equation*}\label{SCATTER2}
u(t)=e^{i(t-t_j)\Delta^2}u(t_j)+i\lambda\int_{t_j}^{t}e^{i(t-s)\Delta^2}(|x|^{-b}|u|^\alpha u)(s)ds.
\end{equation*}
Let us first assume that (i) or (ii) in Theorem \ref{GWPH2} hold. In this case,
from Lemmas \ref{Lemma-Str} and \ref{lemmaglobal1},
\begin{equation}\label{SCATTER3}
\begin{split}
\|u\|_{B(L^2;I_j)}&\leq c\|u(t_j)\|_{L^2_x}+c\left\|\chi_B|x|^{-b}|u|^\alpha u \right\|_{B'(L^2;I_j)}	+ c\left\|\chi_{B^C}|x|^{-b}|u|^\alpha u \right\|_{B'(L^2;I_j)}\\
& \leq c\|u(t_j)\|_{L^2_x}+c\| u\|^{\theta}_{L^\infty_{I_j}H^2_x}\|u\|^{\alpha-\theta}_{B(\dot{H}^{s_c};I_j)} \|  u\|_{B(L^2;I_j)}\\
&\leq c\|u(t_j)\|_{L^2_x} +c\eta^\theta \delta^{\alpha-\theta}\| u\|_{B(L^2;I_j)}.
\end{split}
\end{equation}	 
Also, from Proposition \ref{estimativanaolinear} and Lemmas \ref{lemmaglobal2} and \ref{lemmaparcial},
\begin{equation}\label{SCATTER4}
\begin{split}
\|\Delta u\|_{B(L^2;I_j)}&\leq c\|\Delta u(t_j)\|_{L^2_x}+c\|\nabla(|x|^{-b}|u|^\alpha u)\|_{L^2_{I_j}L_x^{\frac{2N}{N+2}}}\\
& \leq c\|\Delta u(t_j)\|_{L^2_x}+c\| u\|^{\theta}_{L^\infty_{I_j}H^2_x}\|u\|^{\alpha-\theta}_{B(\dot{H}^{s_c};I_j)} \|  \Delta u\|_{B(L^2;I_j)}\\
&\leq c\|\Delta u(t_j)\|_{L^2_x} +c\eta^\theta \delta^{\alpha-\theta}\| \Delta u\|_{B(L^2;I_j)}.
\end{split}
\end{equation}
Thus, \eqref{SCATTER3} and \eqref{SCATTER4} yield
$$
\|u\|_{B(L^2;I_j)}+\|\Delta u\|_{B(L^2;I_j)} \leq c \eta+c\eta^\theta\delta^{\alpha-\theta}(\|u\|_{B(L^2;I_j)}+\|\Delta u\|_{B(L^2;I_j)}).
$$
Consequently, by taking   $\delta>0$ such that $ \eta^\theta\delta^{\alpha-\theta}<\frac{1}{2c}$ we deduce
$$
\| u\|_{B(L^2;I_j)}+	\|\Delta u\|_{B(L^2;I_j)} \leq 2c\eta.
$$
By summing over the $n$ intervals, we conclude   \eqref{SCATTER1} if (i) or (ii) in Theorem \ref{GWPH2} hold. If we assume that (iv) holds, in view of Lemma \ref{lemmaglobal3} and taking into account that 	$\|u\|_{L^\infty_{I_j}H^2_x}\leq \|\Delta u\|_{B(L^2;I_j)}+\|u\|_{B(L^2;I_j)}$, the analysis is similar to that given above.

It remains to establish \eqref{SCATTER1} if (iii) holds. The argument here is a little bit different in view of the term $\|\Delta u\|^{\alpha-\theta^*+\sigma}_{B(L^2)}$ appearing in Lemma \ref{lemmaglobal4}. Indeed, for $t\in I_j$, let us set
$$
A(t)=\| u\|_{B(L^2;[t_j,t])}+	\|\Delta u\|_{B(L^2;[t_j,t])}.
$$
As in \eqref{SCATTER3}, we have
\begin{equation}\label{SCATTER31}
\begin{split}
\|u\|_{B(L^2;[t_j,t])}& \leq c\|u(t_j)\|_{L^2_x}+c\| u\|^{\theta}_{L^\infty_{[t_j,t]}H^2_x}\|u\|^{\alpha-\theta}_{B(\dot{H}^{s_c};[t_j,t])} \|  u\|_{B(L^2;[t_j,t])}\\
&\leq c\|u(t_j)\|_{L^2_x} +c\eta^\theta \delta^{\alpha-\theta}A(t).
\end{split}
\end{equation}
On the other hand, by Proposition \ref{estimativanaolinear} and Lemma \ref{lemmaglobal4},
\begin{equation}\label{SCATTER41}
\begin{split}
\|\Delta u\|_{B(L^2;[t_j,t])}&\leq c\|\Delta u(t_j)\|_{L^2_x}+c\|\nabla(|x|^{-b}|u|^\alpha u)\|_{L^2_{[t_j,t]}L_x^{\frac{2N}{N+2}}}\\
&\leq c\|\Delta u(t_j)\|_{L^2_x} + c\| u\|^{\theta}_{L^\infty_{[t_j,t]}H^2_x}\|u\|^{\alpha-\theta}_{B(\dot{H}^{s_c};[t_j,t])}\left(\| \Delta u\|_{B(L^2;[t_j,t])}+\| u\|_{B(L^2;[t_j,t])}\right) \\
& \quad +c\| u\|^{1-\sigma}_{L_{[t_j,t]}^\infty H^2_x}\|u\|^{\theta^*}_{B(\dot{H}^{s_c};[t_j,t])} \|\Delta u\|^{\alpha-\theta^*+\sigma}_{B(L^2;[t_j,t])}\\
& \leq c\|\Delta u(t_j)\|_{L^2_x}+c\eta^{\theta}\delta^{\alpha-\theta}A(t) + c\eta^{1-\sigma}\delta^{\theta^*}A(t)^{\alpha-\theta^*+\sigma}.
\end{split}
\end{equation}
By summing \eqref{SCATTER31} and \eqref{SCATTER41} we get
\begin{equation}\label{SCATTER42}
A(t)\leq c\eta+2c\eta^{\theta}\delta^{\alpha-\theta}A(t) + c\eta^{1-\sigma}\delta^{\theta^*}A(t)^{\alpha-\theta^*+\sigma}.
\end{equation}
We first choose  $\delta$ sufficiently small such that $2c\eta^{\theta}\delta^{\alpha-\theta}<\frac{1}{2}$ to obtain, from \eqref{SCATTER42},
\begin{equation}\label{SCATTER43}
A(t)\leq 2c\eta+2c\eta^{1-\sigma}\delta^{\theta^*}A(t)^{\alpha-\theta^*+\sigma}, \qquad t\in I_j.
\end{equation}
By noting that $\alpha-\theta^*+\sigma>1$, if $\delta$ is sufficiently small, a standard continuity argument shows that $A(t)\leq 2c\eta$, for any $t\in I_j$. Since $A(t)$ is bounded on $I_j$ we conclude that $\| u\|_{B(L^2;I_j)}+	\|\Delta u\|_{B(L^2;I_j)}$ is finite. By summing over the $n$ intervals, we finally obtain \eqref{SCATTER1}.
	
Now we turn attention back to the   proof of the proposition. The proof is quite standard by now. Indeed, assume that (i) or (ii) in Theorem \ref{GWPH2} hold and let
$$
\phi^+=u_0+i\lambda\int\limits_{0}^{+\infty}e^{i(-s)\Delta^2}|x|^{-b}(|u|^\alpha u)(s)ds.
$$
We claim that $\phi^+ \in H^2(\mathbb{R}^N)$. To see this, following the above steps, we get 
\begin{equation*}
 \|\phi^+\|_{L^2}\leq c\|u_0\|_{L^2}+c \| u \|^\theta_{L^\infty_tH^2_x}\| u \|^{\alpha-\theta}_{B(\dot{H}^{s_c})}\|u\|_{B(L^2)}
\end{equation*}	 
and
\begin{equation*}
\|\Delta \phi^+\|_{L^2}\leq  c\|\Delta u_0\|_{L^2}+c\| u\|^{\theta}_{L^\infty_tH^2_x}\|u\|^{\alpha-\theta}_{B(\dot{H}^{s_c})} (\| \Delta u\|_{B(L^2)} +\|  u\|_{B(L^2)}). 
\end{equation*}
Therefore, \eqref{SCATTER1} yields the claim.
  
Since $u$ is a solution of \eqref{IBNLS}, a simple inspection gives
$$
 u(t)-e^{it\Delta^2}\phi^+=-i\int\limits_{t}^{+\infty}e^{i(t-s)\Delta^2}|x|^{-b}(|u|^\alpha u)(s)ds.
$$
Hence, as above,
$$
 \|u(t)-e^{it\Delta^2}\phi^+\|_{L^2_x}\leq c \| u \|^\theta_{L^\infty_tH^2_x}\| u \|^{\alpha-\theta}_{B(\dot{H}^{s_c};[t,\infty))}\|u\|_{B(L^2)}
$$
and
\begin{align*}
\|\Delta(u(t)-e^{it\Delta^2}\phi^+)\|_{L^2_x}    &\leq  c  \| u\|^{\theta}_{L^\infty_tH^2_x}\|u\|^{\alpha-\theta}_{B(\dot{H}^{s_c};[t,\infty))} (\| \Delta u\|_{B(L^2)} +\|  u\|_{B(L^2)})
\end{align*}
Now, observing that $\|u\|_{B(\dot{H}^{s_c};[t,\infty))}\rightarrow 0$ as $t \rightarrow +\infty$, using \eqref{SCATTER1}, we conclude that 
\begin{equation}\label{scat1}
\|u(t)-e^{it\Delta^2}\phi^+\|_{H^2_x}\rightarrow 0, \,\,\textnormal{as}\,\,t\rightarrow +\infty.
\end{equation}

By using similar arguments one may also see that 
$
\|u(t)-e^{it\Delta^2}\phi^-\|_{H^2_x}\rightarrow 0, \,\,\textnormal{as}\,\,t\rightarrow -\infty.
$
where
$$
\phi^-=u_0+i\lambda\int_0^{-\infty}e^{i(-s)\Delta^2}|x|^{-b}(|u|^\alpha u)(s)ds.
$$
Thus, the proof of the proposition is completed in this case.
 Let us point out that the crucial points to obtain \eqref{scat1} were \eqref{SCATTER1} and the fact $\|u\|_{B(\dot{H}^{s_c};[t,\infty))}\rightarrow 0$, as $t \rightarrow +\infty$. Since the   norm $\|u\|_{B(\dot{H}^{s_c};[t,\infty))}$ also appears in our estimates if we assume (iii) or (iv) in Theorem \ref{GWPH2}, the proof in this cases follows in a similar fashion as above. So, we omit the details.
\end{proof}

\section{stability}

In this section, we shall show Theorem \ref{LTP}. To this end, we start with the following proposition.

\begin{proposition}\label{STP}{\bf (Short-time perturbation).} 
Assume that assumptions in Theorem \ref{GWPH2} hold.  Let $I\subseteq \mathbb{R}$ be a time interval containing zero and let $\widetilde{u}$ be a solution of
\begin{equation*}\label{PE}
i\partial_t \widetilde{u} +\Delta^2 \widetilde{u} + \lambda|x|^{-b} |\widetilde{u}|^\alpha \widetilde{u} =e,
\end{equation*}  
 defined on $I\times \mathbb{R}^N$, with initial data $\widetilde{u}_0\in H^2(\mathbb{R}^N)$, and satisfying 
\begin{equation}\label{PC11}  
\sup_{t\in I}  \|\widetilde{u}(t)\|_{H^2_x}\leq M \;\; \mbox{and}\;\; \|\widetilde{u}\|_{B(\dot{H}^{s_c}; I)}\leq \varepsilon,
\end{equation}
for some positive constant $M$ and some small $\varepsilon>0$.
	
\indent  Let $u_0\in H^2(\mathbb{R}^N)$ be such that 
\begin{equation}\label{PC22}
\|u_0-\widetilde{u}_0\|_{H^2}\leq M'\;\; \mbox{and}\;\; \|e^{it\Delta^2}(u_0-\widetilde{u}_0)\|_{B(\dot{H}^{s_c}; I)}\leq \varepsilon,\;\;\textnormal{for }\; M'>0.
\end{equation}
Assume also that
\begin{equation}\label{PC33}
\|e\|_{B'(L^2; I)}+\|\nabla e\|_{L^2_IL^{\frac{2N}{N+2}}}+  \|e\|_{B'(\dot{H}^{-s_c}; I)}\leq \varepsilon.
\end{equation}
\indent There exists $\varepsilon_0(M,M')>0$ such that if $\varepsilon<\varepsilon_0$, then there is a unique solution $u$ of \eqref{IBNLS} on $I\times \mathbb{R}^N$, with  $u(0)=u_0$,  satisfying 
\begin{equation}\label{C} 
\|u-\widetilde{u}\|_{B(\dot{H}^{s_c}; I)}\lesssim \varepsilon 
\end{equation}
and
\begin{equation}\label{C1}
\|u\|_{B(L^2; I)}+\|\Delta u\|_{B(L^2; I)}\lesssim c(M,M').
\end{equation}
\end{proposition}
\begin{proof} 
We will prove the result by assuming that (i) in Theorem \ref{GWPH2} holds. The other cases are dealt with similarly.  Without loss of generality, we may assume that $0=\inf I$.

We start with the following claim: 

\noindent {\bf Claim}: If $\|\widetilde{u}\|_{B(\dot{H}^{s_c};I)}\leq \varepsilon_0$, for some $\varepsilon_0>0$ enough small, then $\|\widetilde{u}\|_{B(L^2;I)} +|\Delta\widetilde{u}\|_{B(L^2;I)}\lesssim M.$
 Indeed, we will show that 
\begin{equation}\label{widetilde{u}}
\|\Delta \widetilde{u}\|_{B(L^2;I)}\lesssim M.
\end{equation}
Similar estimates also imply $\|\widetilde{u}\|_{B(L^2;I)}\lesssim M$. Since $\widetilde{u}$ satisfies an integral equation similar to that in \eqref{OPERATOR}, we have from Proposition \ref{estimativanaolinear},
$$
\|\Delta \widetilde{u}\|_{B(L^2;I)}\lesssim  \|\Delta \widetilde{u}_0\|_{L^2}+ \left\|\nabla(|x|^{-b}|\widetilde{u}|^\alpha\widetilde{u})\right\|_{L^2_IL_x^{\frac{2N}{N+2}}}+\|\nabla e\|_{L^2_IL_x^{\frac{2N}{N+2}}}.
$$
Furthermore, Lemma \ref{lemmaglobal2}, \eqref{PC11}, and \eqref{PC33} imply
\[
\begin{split}
	\|\Delta\widetilde{u}\|_{B(L^2;I)}&\lesssim M+ \| \widetilde{u} \|^\theta_{L^\infty_tH^2_x}\| \widetilde{u} \|^{\alpha-\theta}_{B(\dot{H}^{s_c};I)}  \|\Delta \widetilde{u} \|_{B(L^2;I)}  +\varepsilon\\
	&\lesssim M+\varepsilon+ M^\theta \varepsilon_0^{\alpha-\theta}\|\Delta \widetilde{u} \|_{B(L^2;I)},
\end{split}
\]
which immediately implies \eqref{widetilde{u}} if $\varepsilon_0$ is sufficiently small.

The solution $u$ will be obtained as    $u=\widetilde{u}+w$, where $w$ is the solution of the following   IVP 	
\begin{equation}\label{IVPP} 
\begin{cases}
i\partial_tw +\Delta^2 w + H(x,\widetilde{u},w)+e= 0,&  \\
w(0,x)= u_0(x)-\widetilde{u}_0(x),& 
\end{cases}
\end{equation}
with $H(x,\widetilde{u},w)=\lambda|x|^{-b} \left(|\widetilde{u}+w|^\alpha (\widetilde{u}+w)-|\widetilde{u}|^\alpha \widetilde{u}\right)$.
It is then suffices to show that \eqref{IVPP} indeed has a solution defined on $I\times \mathbb{R}^N$. To do so, we also use the contraction mapping principle combined with the estimates established in Section \ref{secglo}.
Consider  the map
\begin{equation}\label{IEP} 
G (w)(t):=e^{it\Delta^2}w_0+i  \int_0^t e^{i(t-s)\Delta^2}(H(x,\widetilde{u},w)+e)(s)ds
\end{equation}
and define
$$
B_{\rho,K}=\{ w\in C(I;H^2(\mathbb{R}^N)):\;\|w\|_{B(\dot{H}^{s_c};I)}\leq \rho\;\textnormal{and}\;\|w\|_{B(L^2;I)}+\|\Delta w\|_{B(L^2;I)}\leq K    \},
$$
where $\rho>0$ and $K>0$ will be chosen later.
From Lemma \ref{Lemma-Str} and Proposition \ref{estimativanaolinear}, one has 
\begin{equation}\label{SP1}
\begin{split}
\|G(w)\|_{B(\dot{H}^{s_c};I)}&\lesssim \|e^{it\Delta^2}w_0\|_{B(\dot{H}^{s_c};I)}+ \|\chi_{B} H \|_{B'(\dot{H}^{-s_c};I)}+\|\chi_{B^C} H \|_{B'(\dot{H}^{-s_c};I)}+\|e \|_{B'(\dot{H}^{-s_c};I)}\\
&\lesssim \|e^{it\Delta^2}w_0\|_{B(\dot{H}^{s_c};I)}+ \|\chi_{B} H \|_{L_I^{\widetilde{a}'}L_x^{\widehat{r}'}}+\|\chi_{B^C} H \|_{L_I^{\widetilde{a}'}L_x^{\widehat{r}'}}+\|e \|_{B'(\dot{H}^{-s_c};I)}
\end{split}
\end{equation}
\begin{equation}\label{SP2}
\begin{split}
\|G(w)\|_{B(L^2;I)}&\lesssim \|w_0\|_{L^2}+ \|\chi_{B} H \|_{B'(L^2;I)}+ \|\chi_{B^C} H \|_{B'(L^2;I)}+\|e\|_{B'(L^2;I)}\\
&\lesssim \|w_0\|_{L^2}+ \|\chi_{B} H \|_{L_I^{\widehat{q}'}L_x^{\widehat{r}'}}+ \|\chi_{B^C} H \|_{L_I^{\widehat{q}'}L_x^{\widehat{r}'}}+\|e\|_{B'(L^2;I)}
\end{split}
\end{equation}
and	
\begin{equation}\label{SP3}
\|\Delta G(w)\|_{B(L^2;I)}\lesssim  \|\Delta w_0\|_{L^2}+ \| \nabla H\|_{L^2_IL_x^{\frac{2N}{N+2}}}+\|\nabla e\|_{L^2_IL_x^{\frac{2N}{N+2}}},
\end{equation}	
where the pairs $(\widehat{q},\widehat{r})$ and $(\widetilde{q},\widehat{r})$ were defined in the proof of Lemma \ref{lemmaglobal1}.
Now, from the mean value theorem (see \eqref{FEI}), we easily see that
\begin{equation}\label{EI} 
\left| |\widetilde{u}+w|^\alpha(\widetilde{u}+w)-|\widetilde{u}|^\alpha\widetilde{u} \right|\lesssim |\widetilde{u}|^\alpha|w|+|w|^{\alpha+1}.
\end{equation}
 Lemma \ref{lemmaglobal1} combined with \eqref{EI} give
\begin{align}
\|\chi_{B} H \|_{L_I^{\widetilde{a}'}L_x^{\widehat{r}'}}+\|\chi_{B^C} H \|_{L_I^{\widetilde{a}'}L_x^{\widehat{r}'}}\lesssim  \left(\| \widetilde{u} \|^\theta_{L^\infty_tH^2_x}\| \widetilde{u} \|^{\alpha-\theta}_{B(\dot{H}^{s_c};I)}+ \| w \|^\theta_{L^\infty_tH^2_x} \| w\|^{\alpha-\theta}_{B(\dot{H}^{s_c};I)}    \right) \| w \|_{B(\dot{H}^{s_c};I)}
\end{align}
and
\begin{align}
 \|\chi_{B} H \|_{L_I^{\widehat{q}'}L_x^{\widehat{r}'}}+ \|\chi_{B^C} H \|_{L_I^{\widehat{q}'}L_x^{\widehat{r}'}}\lesssim \left(\| \widetilde{u} \|^\theta_{L^\infty_tH^2_x}\| \widetilde{u} \|^{\alpha-\theta}_{B(\dot{H}^{s_c};I)} + \| w \|^\theta_{L^\infty_tH^2_x}\| w\|^{\alpha-\theta}_{B(\dot{H}^{s_c};I)} \right)\| w \|_{B(L^2;I)}.
\end{align}
Let us now estimate $\|\nabla H\|_{L^2_IL_x^{\frac{2N}{N+2}}}$. From \eqref{SECONDEI} we deduce
\begin{equation*} 
|\nabla H(x,\widetilde{u},w)| \lesssim |x|^{-b-1}(|\widetilde{u}|^{\alpha}+|w|^{\alpha})|w|+|x|^{-b}(|\widetilde{u}|^\alpha+|w|^\alpha) |\nabla w| +E,
 \end{equation*}
 where
\begin{eqnarray*} 
 E &\lesssim& \left\{\begin{array}{cl}
 |x|^{-b}\left(|\widetilde{u}|^{\alpha-1}+|w|^{\alpha-1}\right)|w||\nabla \widetilde{u}|, & \textnormal{if}\;\;\;\alpha> 1 \vspace{0.2cm} \\ 
|x|^{-b}|\nabla \widetilde{u}||w|^{\alpha}, & \textnormal{if}\;\;\;\alpha\leq 1.
\end{array}\right.
\end{eqnarray*}
A consequence of Lemma  \ref{lemmaglobal2} is that
$$
\left\| |x|^{-b-1}|u|^\alpha v  \right\|_{B'(L^2)}\lesssim  \| u\|^{\theta}_{L^\infty_tH^2_x}\|u\|^{\alpha-\theta}_{B(\dot{H}^{s_c})}  \|\Delta v\|_{B(L^2)}.
$$
Therefore,
$$
\|\nabla H\|_{L^2_IL_x^{\frac{2N}{N-2}}} \lesssim \left(\| \widetilde{u} \|^\theta_{L^\infty_tH^2_x}\| \widetilde{u} \|^{\alpha-\theta}_{B(\dot{H}^{s_c};I)} + \| w \|^\theta_{L^\infty_tH^2_x}\| w\|^{\alpha-\theta}_{B(\dot{H}^{s_c};I)} \right)\|\Delta w \|_{B(L^2;I)}+E_1,
$$
where (using Remark \ref{RGP})
\begin{align*}
E_1\lesssim & \left\{\begin{array}{cl}
\left(\| \widetilde{u} \|^\theta_{L^\infty_tH^2_x} \| \widetilde{u} \|^{\alpha-1-\theta}_{B(\dot{H}^{s_c};I)} + \| w \|^\theta_{L^\infty_tH^2_x} \| w \|^{\alpha-1-\theta}_{B(\dot{H}^{s_c};I)} \right) \| w \|_{B(\dot{H}^{s_c};I)}  \|\Delta \widetilde{u} \|_{B(L^2;I)},&\alpha> 1 \vspace{0.2cm} \\
\| w \|^\theta_{L^\infty_tH^2_x}\| w\|^{\alpha-\theta}_{B(\dot{B}^{s_c};I)} \|\Delta \widetilde{u} \|_{B(L^2;I)}\;,\;\;\;\;\;\;\;\;\;\;\alpha\leq 1.
\end{array}\right.
\end{align*}
 
Gathering together the above estimates with our assumptions, we get for any $w\in B_{\rho,K}$,
\begin{equation}\label{SP7}
\|\chi_{B} H \|_{L_I^{\widetilde{a}'}L_x^{\widehat{r}'}}+\|\chi_{B^C} H \|_{L_I^{\widetilde{a}'}L_x^{\widehat{r}'}}\lesssim \left(M^\theta\varepsilon^{\alpha-\theta}+K^\theta \rho^{\alpha-\theta}\right)\rho,
\end{equation}
\begin{equation}\label{SP8}
 \|\chi_{B} H \|_{L_I^{\widehat{q}'}L_x^{\widehat{r}'}}+ \|\chi_{B^C} H \|_{L_I^{\widehat{q}'}L_x^{\widehat{r}'}}\lesssim \left(M^\theta\varepsilon^{\alpha-\theta}+K^\theta \rho^{\alpha-\theta}\right)K,
\end{equation}
and
\begin{align}\label{SP9}
 \|\nabla H\|_{L^2_IL_x^{\frac{2N}{N+2}}} \lesssim \left(M^\theta\varepsilon^{\alpha-\theta}+K^\theta \rho^{\alpha-\theta}\right)K +E_1,
\end{align}
where
\begin{eqnarray*}
E_1&\lesssim & \left\{\begin{array}{cl}
\left( M^\theta \varepsilon^{\alpha-1-\theta} + K^\theta \rho^{\alpha-1-\theta} \right) \rho M,&\textnormal{if}\;\;\alpha > 1, \vspace{0.2cm} \\
K^\theta \rho^{\alpha-\theta} M,\;\;\;\;\;\; \textnormal{if}\;\;\;\;\;\alpha\leq 1.
\end{array}\right.
\end{eqnarray*}
Hence, it follows from \eqref{SP1}-\eqref{SP3} and assumptions \eqref{PC22}-\eqref{PC33} that 
$$
\|G(w)\|_{B(\dot{H}^{s_c};I)}\leq  c\varepsilon+ cK_1\rho,
$$
$$
\|G(w)\|_{B(L^2;I)}\leq cM'+c\varepsilon +cK_1K,
 $$
 \begin{equation*}
 \|\Delta G(w)\|_{B(L^2;I)}\leq cM'+c\varepsilon +cK_1K+cK_2 \rho M, \qquad \mbox{if}\;\;\alpha> 1,
 \end{equation*}
 and
 \begin{equation*}
 \|\Delta G(w)\|_{B(L^2;I)}\leq cM'+c\varepsilon + cK_1K + K^\theta \rho^{\alpha-\theta} M, \qquad \mbox{if}\;\;\alpha\leq 1,
 \end{equation*}
 where
 $K_1=M^\theta\varepsilon^{\alpha-\theta}+K^\theta \rho^{\alpha-\theta}$ and $K_2= M^\theta \varepsilon^{\alpha-1-\theta} + K^\theta \rho^{\alpha-1-\theta}$.
By choosing $\rho=2c\varepsilon$, $K=3cM'$ and $\varepsilon_0$ sufficiently small such that 
$$
cK_1<\frac{1}{3}\;\;\;\;\textnormal{and}\;\;\;c(\varepsilon+K_2\rho M+K^\theta \rho^{\alpha-\theta} M)<\frac{K}{3},
$$ 
we have
\begin{equation*}
\|G(w)\|_{B(\dot{H}^{s_c};I)}\leq \rho\;\;\;\textnormal{and}\;\;\;\|G(w)\|_{B(L^2;I)}+\|\Delta G(w)\|_{B(L^2;I)}\leq K.
\end{equation*}
Therefore, $G$ is well defined and maps $B_{\rho,K}$ into itself. By using a similar argument we can also show that $G$ is a contraction. Thus, from the contraction mapping principle we obtain a unique solution $w$ on $I\times \mathbb{R}^N$ such that 
$$
\|w\|_{B(\dot{H}^{s_c};I)}\lesssim \varepsilon \;\;\;\textnormal{and}\;\;\;\|w\|_{B(L^2;I)}+\|\Delta w\|_{B(L^2;I)} \lesssim c(M,M'),
$$ 
which it turn implies \eqref{C} and \eqref{C1}. This completes the proof of the proposition.
\end{proof}

\begin{remark}\label{RSP} 
From \eqref{SP7}-\eqref{SP8}, we also obtain the following estimates:
\begin{equation}\label{RSP1}
\|\chi_{B}H(x,\widetilde{u},w)\|_{B'(\dot{H}^{-s_c}; I)}+\|\chi_{B^C}H(\cdot,\widetilde{u},w)\|_{B'(\dot{H}^{-s_c}; I)}\leq c(M,M') \varepsilon
\end{equation}
and
\begin{equation}\label{RSP2}
\|\chi_{B}H(x,\widetilde{u},w)\|_{B'(L^2; I)}+\|\chi_{B^C}H(x,\widetilde{u},w)\|_{B'(L^2; I)}+\|\nabla H(\cdot,\widetilde{u},w)\|_{L^2_IL_x^{\frac{2N}{N-2}}}\leq c(M,M')\varepsilon^{\alpha-\theta}.
\end{equation}
\end{remark}

\ Next, in view of the previous proposition we are able to show Theorem \ref{LTP}. The idea is to  iterate the short-time perturbation result.

\begin{proof}[\bf {Proof of Theorem \ref{LTP}}] The proof is similar to that in \cite[Proposition 4.9]{paper2}; so we give only the main steps. As before, we can assume $0=\inf I$.  Since $\|\widetilde{u}\|_{B(\dot{H}^{s_c}; I)}\leq L$,  we may take a partition of $I$ into $n = n(L,\varepsilon)$ intervals $I_j = [t_j ,t_{j+1}]$ such that $\|\widetilde{u}\|_{B(\dot{H}^{s_c};I_j)}\leq \varepsilon$, where  $\varepsilon<\varepsilon_0(M,2M')$ and $\varepsilon_0$ is given in Proposition \ref{STP}.
Since, on $I_j$,
  \begin{equation*}
   w(t)=e^{i(t-t_j)\Delta^2}w(t_j)+i\int_{t_j}^{t}e^{i(t-s)\Delta^2}(H(x,\widetilde{u},w)+e)(s)ds,
  \end{equation*}
solves the equation in \eqref{IVPP} with initial data $w(t_j)=u(t_j)-\widetilde{u}(t_j)$, by choosing $\varepsilon_1=\varepsilon_1(n,M,M')$ sufficiently small we may reiterate Proposition \ref{STP} to obtain, for each $0\leq j<n$ and $\varepsilon<\varepsilon_1$, 
\begin{equation}\label{LP1}
\|u-\widetilde{u}\|_{B(\dot{H}^{s_c};I_j)}\leq c(M,M',j)\varepsilon
\end{equation}
and
\begin{equation}\label{LP2}
\|w\|_{B(\dot{H}^{s_c};I_j)}+\|w\|_{B'(L^2;I_j)}+\|\Delta w\|_{B'(L^2;I_j)}\leq c(M,M',j),
\end{equation}
provided that (for each $0\leq j<n$)
\begin{equation}\label{LP3}
 \|e^{i(t-t_j)\Delta^2}(u(t_j)-\widetilde{u}(t_j))\|_{B(\dot{H}^{s_c};I_j)}\leq c(M,M',j)\varepsilon\leq \varepsilon_0
 \end{equation}
 and
 \begin{equation}\label{LP4}
 \|u(t_j)-\widetilde{u}(t_j)\|_{H^2_x}\leq 2M'.
 \end{equation}
By summing \eqref{LP1} and \eqref{LP2} over all subintervals $I_j$, we get the desired. 
 
It remains to establish \eqref{LP3} and \eqref{LP4}.  But from Lemma \ref{Lemma-Str},
\[
\begin{split}
 \|e^{i(t-t_j)\Delta^2}w(t_j)\|_{B(\dot{H}^{s_c};I_j)}&\lesssim \|e^{it\Delta^2}w_0\|_{B(\dot{H}^{s_c}; I)}+\|\chi_{B}H(x,\widetilde{u},w)\|_{B'(\dot{H}^{-s_c};[0,t_j])}\\
  &\quad +\|\chi_{B^C}H(x,\widetilde{u},w)\|_{B'(\dot{H}^{-s_c};[0,t_j])}+\|e\|_{B'(\dot{H}^{-s_c};I)},
  \end{split}
  \]
which  by \eqref{RSP1} and an inductive argument  yield
$$
\|e^{i(t-t_j)\Delta^2}(u(t_j)-\widetilde{u}(t_j))\|_{B(\dot{H}^{s_c}; I_j)}\lesssim \varepsilon+\sum_{k=0}^{j-1}c(M,M',k)\varepsilon.
$$
  
By a similar argument but now using  \eqref{RSP2} we get
\[
 \begin{split}
 \|u(t_j)-\widetilde{u}(t_j)\|_{H^2_x}&\lesssim  \|u_0-\widetilde{u}_0\|_{H^2}+\|e\|_{B'(L^2; I)}+\|\nabla e\|_{L^2_IL^\frac{2N}{N+2}_x}\\
 &\quad+\| H(x,\widetilde{u},w)\|_{B'(L^2;[0,t_j])}+\|\nabla H(x,\widetilde{u},w)\|_{B'(L^2;[0,t_j])}\\
 &\lesssim M'+\varepsilon+\sum_{k=0}^{j-1}C(k,M,M')\varepsilon^{\alpha-\theta}.
 \end{split}
 \]
 Taking $\varepsilon_1$ sufficiently small, we see that \eqref{LP3} and \eqref{LP4} hold. This completes the proof of the theorem.
 \end{proof} 

\section*{Acknowledgments} 
Most part of this work was done when the first author had a postdoctoral position at Universidade Federal de Minas Gerais (UFMG). C.M.G. was supported by CAPES/Brazil. A.P. was partially supported by CNPq/Brazil grants 402849/2016-7 and 303098/2016-3 and FAPESP/Brazil grant 2019/02512-5.

\bibliographystyle{abbrv}
\bibliography{bibguzman}	

\end{document}